\documentclass[reqno]{amsart}
\usepackage{a4wide,enumerate,xcolor,graphicx}
\usepackage{color}
\usepackage{enumitem}
\usepackage{hyperref}
\usepackage{amsmath}
\usepackage{dsfont} 
\allowdisplaybreaks
\let\pa\partial

\let\eps\varepsilon
\newcommand{\N}{{\mathbb N}}
\newcommand{\R}{{\mathbb R}}

\newcommand{\diver}{\operatorname{div}}

\newcommand{\E}{{\mathcal E}}

\newcommand{\de}{\textnormal{d}}
\newcommand{\rb}{\bar\rho}

\newtheorem{theorem}{Theorem}
\newtheorem{lemma}[theorem]{Lemma}

\newtheorem{remark}[theorem]{Remark}

\newtheorem{assumption}[theorem]{Assumption}

\title[Relative entropy in moderate regime]{Quantitative convergence in relative entropy for a moderately interacting particle system on $\R^d$}

\author[L. Chen]{Li Chen}
\address{University of Mannheim, School of Business Informatics and Mathematics, 
	68131 Mann\-heim, Germany}
\email{chen@math.uni-mannheim.de}

\author[A. Holzinger]{Alexandra Holzinger$^1$}
\address{Mathematical Institute, University of Oxford, Woodstock Road, Oxford, OX2 6GG, England, United Kingdom}
\email[Corresponding Author]{alexandra.holzinger@maths.ox.ac.uk}

\author[X. Huo]{Xiaokai Huo}
\address{Department of Mathematics, Iowa State University, 
	411 Morrill Road, Ames, IA 50014, United States}
\email{xhuo@iastate.edu}

\date{\today}

\thanks{LC has been supported by the DF grant CH955/8-1. AH has been supported by the Advanced Grant Nonlocal-CPD (Nonlocal PDEs for Complex Particle Dynamics: Phase Transitions, Patterns and Synchronization) of the European Research Council Executive Agency (ERC) under the European Union's Horizon 2020 research and innovation programme (grant agreement No. 883363) and under the European Union's Horizon 2020 research and innovation programme, ERC Advanced Grant no. 101018153. Additionally, AH acknowledges partial support from the Austrian Science Fund (FWF), grants P33010 and F65.}


\begin{document}
\maketitle
\setcounter{footnote}{1}
\footnotetext{Corresponding author.}
\begin{abstract}
    This article shows how to combine the relative entropy method in \cite{jabin2018quantitative, bresch2019mean} and the regularized $L^2(\R^d)$-estimate in \cite{oeschlager1987fluctuation} to prove a strong propagation of chaos result for the viscous porous medium equation from a moderately interacting particle system in $L^\infty(0,T; L^1(\R^d))$-norm. In the moderate interacting setting, the interacting potential is a smoothed Dirac Delta distribution, however, current results regarding the relative entropy methods for singular potentials do not apply. The propagation of chaos result for the porous media equation holds on $\R^d$ for any dimension $d\geq 1$ and provides a quantitative result where the rate of convergence depends on the moderate scaling parameter and the dimension $d\geq 1$. Additionally, the presented method can be adapted for moderately interacting systems for which the convergence probability in \cite{LP,LCZ} holds -- thus a propagation of chaos result in relative entropy can be obtained for kernels approximating Coulomb potentials.
    \end{abstract}

    \section{Introduction} 
    
    The aim of this article is to derive a quantitative convergence result in \newline $L^\infty(0,T;L^1(\R^d))$-norm for the following viscous porous medium equation 
    \begin{equation}\label{eq:2}
    \partial_t \rb = \frac12 \diver\big((1+\rb)\nabla \rb\big),\mbox{ for }(t,x) \in (0,T)\times \mathbb{R}^d,\quad \rb\big|_{t=0} = \rho_0,\mbox{ in } \mathbb{R}^d,
    \end{equation}
    for $T>0$ from a moderately interacting particle system of mean-field type with algebraic scaling of the moderate interaction parameter. Equation \eqref{eq:2} is considered for arbitrary dimensions $d\geq 1$. The underlying method of this article is a combination of the relative entropy method for mean-field limits recently developed in \cite{jabin2018quantitative, bresch2019mean, serfaty2020mean} for weakly interacting particle systems and $L^2(\R^d)$-norm estimates for the smoothed empirical measure which were originally developed by Oelschl\"ager in \cite{oeschlager1987fluctuation} in order to prove a fluctuation theorem for moderately interacting particles. Hence, this article is also aimed to show a connection between the relative entropy method and convergence results in regularised $L^2(\R^d)$-norm.  
    To the best of the authors knowledge, it is the first time to show a (quantitative) propagation of chaos result by relative entropy estimates in the setting of moderately interacting particle systems. 
    
    It is well-known (e.g. \cite{figalli2008convergence, oelschlager1985law}) that the moderately interacting particle system corresponding to \eqref{eq:2} for $N \in \N$ interacting particles on $\R^d$ reads as
    \begin{eqnarray}
    \label{eq:1}
        &&\de X_{i,t}^{N,\eta} = -\frac{1}{2N} \sum_{j=1,j\neq i}^{N} \nabla V^\eta(X_{i,t}^{N,\eta} -X_{j,t}^{N,\eta} ) \de t +  \de B_t^i, \\
        \nonumber && X_{i,0}^{N,\eta}  = \xi_i,\quad i=1,\ldots,N,
    \end{eqnarray}
    where the interaction kernel $V^\eta \geq 0$ depends on the number of particles $N\in \N$ via the moderate interaction parameter $\eta = \eta(N)$. $(B_t^i)_{i=1}^N$ are independent Brownian motions on $\R^d$ and $(\xi_i)_{i=1}^N$ are assumed to be i.i.d random variables. Concerning the moderate interaction parameter $\eta>0$, we consider an algebraic scaling as in \cite{oelschlager1985law}, i.e. $\eta=N^{-\beta/d}$ for $\beta >0$. The interacting potential $V^\eta$ is assumed to approximate the Dirac distribution for $\eta \to 0$ in the following way: 
    \begin{equation}
    	\label{eq:V1} V^\eta(x)=W^\eta*W^\eta(x),\quad W^\eta(x)=\eta^{-d}W\Big(\frac{x}{\eta}\Big) \mbox{ for }x\in \mathbb{R}^d,
    \end{equation}
    where $W$ is a non-negative radially symmetric function with $\|W\|_{L^1(\R^d)}=1$. Clearly, by setting $V:= W \ast W$ we get  $\|V\|_{L^1(\R^d)}=1$, $\nabla V(0)=\mathbf{0}$ and $V^\eta = \eta^{-d} V(x/\eta)$. Recalling the algebraic connection between $\eta$ and $N$, it is clear that $V^\eta$ forms a smooth approximation of the Dirac distribution for $N \to \infty$ (which implies $\eta \to 0$). Due to the non-negativity of $V$ particle system \eqref{eq:1} describes a repulsive regime in the large population limit.
    
     In the moderate regime the strength of interaction is stronger than in the classical mean-field regime ($O(N^{-1})$) but weaker than in the strong regime ($O(1)$). The interaction potential is scaled such that it approximates a Dirac distribution, which allows to derive local equations in the limit. For a more detailed introduction to moderate regimes, we refer the reader to \cite{oelschlager1985law} and \cite[Chapter II.]{Sznitman91}.
    
    \subsection{State of the art} Mean field limits of large particle systems is an active field of research, hence it is out of the scope of this article to make a full review of it. Instead, we refer to \cite{ChaintronDiez22I,  ChaintronDiez21II, golse2016dynamics, jabin2017mean} and the references therein for the recent developments. In the following we will focus only on the results of moderately interacting systems and on the recently developed relative entropy method for mean-field limits. 
    
    A rigorous proof of a limit for a moderately interacting particle system has been first given by Oelschl\"ager in \cite{oelschlager1985law} to obtain the viscous porous medium equation \eqref{eq:2} under algebraic scaling $\eta=N^{-\beta/d}$ for $\beta \in (0, d/(d+2))$. This result was generalized in \cite{Meleard87} by using a different notion of convergence. The derivation of the porous media equation without viscosity, i.e. without linear diffusion $\frac{1}{2}\Delta \rb$, was shown in \cite{oelschlager1990large} using a deterministic particle system under more restrictive conditions on the dimension and the initial datum $\rho_0$. Both articles \cite{oelschlager1985law,oelschlager1990large} prove propagation of chaos in a `weak' sense, i.e. on the level of empirical measures. 
     In \cite{philipowski2007interacting}, the mean-field limit towards the porous medium equation without viscosity is shown under logarithmic scaling ($\eta=(\log(N)))^{-\alpha}$ for some $\alpha >0$) by using a stochastic particle system similar to \eqref{eq:1} where the diffusion coefficient tends to zero allowing for a broader class of initial data than in \cite{oelschlager1990large}. For fixed diffusion coefficient, the convergence towards the viscous porous media equation is shown in a `strong' sense by coupling methods, whereas the full limit with vanishing viscosity holds in a weak sense. The above results are all concerned with $\Delta \rb^2$ with exponent $2$. For general exponents, the derivation of the porous medium equation from moderate particle systems is given in \cite{figalli2008convergence} under logarithmic scaling of $\eta>0$ by coupling methods using a similar strategy as in \cite{philipowski2007interacting}.
    Changing the sign of the non-linear diffusion term in \eqref{eq:2} to  $-\Delta \rb^2$ leads to a PDE of diffusion-aggregation type. The well-posedness analysis on a bounded domain this PDE has been studied in \cite{chen2021class}.  The derivation of this diffusion-aggregation equation from a moderately interacting particle system was given in \cite{chen2018modeling} by coupling methods but with logarithmic scaling of $\eta>0$. The above approaches on the single species case have recently been extended to multi-species systems with cross-diffusion, see for example \cite{chen2019rigorous,chen2021rigorous}. 
    Concerning deterministic approximating sequences of \eqref{eq:2}, in \cite{oelschlager2001viscous}, Oelschl\"ager showed a convergence result of McKean-Vlasov integro-differential equations towards the viscous porous media equation in the framework of $C^\infty(\R^d)$ solutions. The estimates are similar to those used in Section \ref{sec:intermediate}, however, in order to be able to fill the gap towards the relative entropy method, we provide estimates in $H^s(\R^d)$-norm. Recently, in \cite{Olivera20} and \cite{Olivera20Burgers} the authors developed new quantitative mean-field estimates for classes of moderately interacting particle systems, which do not include the viscous porous media equation, by using a semi-group approach.
    
    In 1991, Oelschl\"ager \cite{oelschlager91} already used energy-based estimates for moderate interacting second order systems. The relative entropy method for mean-field limits was applied in \cite{jabinWang2016} in order to show convergence of the marginals of the $N$-particle distribution in $L^1(\R^d)$-norm in the setting of Vlasov equation. Following this work quantitative mean-field limits with $W^{-1,\infty}$ interaction potentials have been studied in \cite{jabin2018quantitative} as a consequence of relative entropy estimates on the torus $\mathbb{T}^d$.
    Soon after, the method of modulated free energy was introduced in \cite{serfaty2018systems,serfaty2020mean} to study the mean-field behaviour of particle systems with Coulomb-type interaction potentials. A successful combination of the techniques in \cite{serfaty2018systems,serfaty2020mean} and \cite{jabin2018quantitative} made it possible to obtain a propagation of chaos result in $L^1(\mathbb{T}^d)$-norm for mean-field systems with logarithmic interaction potentials in arbitrary dimensions which includes the Keller-Segel system in dimension $2$, see \cite{bresch2019mean}.
    Recently, in \cite{Lacker23} Lacker developed new techniques based on the well-known BBGKY hierarchy in order to derive improved rates for the convergence of the $k$-th marginals of the joint law of interacting particle systems of mean-field type. The techniques in \cite{Lacker23} deal with a class of weakly interacting particles. An extension of this technique for moderate interaction is still missing in the literature.
    
    \subsection{Aim and structure of the article} 
    In this article, we aim to derive the viscous porous media equation \eqref{eq:2} from the mean-field type moderate interacting particle system \eqref{eq:1} by proving a quantitative propagation of chaos result for the marginals of the $N$-particle distribution in $L^\infty(0,T;L^1(\R^d))$-norm for any $T>0$. Inspired by \cite{jabin2018quantitative, bresch2019mean}, we use relative entropy estimates and the Csisz\'ar-Kullback-Pinsker inequality to derive the desired estimates. In the moderate interacting setting, the interacting potential is a smoothed Dirac Delta distribution, therefore, the existing results concerning relative entropy estimates in mean-field settings can not be applied. 
    
    We introduce an intermediate nonlocal McKean-Vlasov system, and split the proof in two parts, including estimates between this intermediate system and \eqref{eq:2}, as well as a propagation of chaos result from particle system \eqref{eq:1} to the intermediate McKean-Vlasov system. Both parts include estimates in relative entropy. In order to derive those estimates, we apply a regularised $L^2(\R^d)$ estimate which was originally developed for the study of fluctuations around the mean-field limit in \cite{oeschlager1987fluctuation}. This avoids the application of technical challenging large deviation estimates as in \cite{bresch2019mean,jabin2017mean} and allows us to give a quantitative mean-field convergence result for the marginals of the $N$-particle distribution in the moderate setting. The convergence rate depends on the mean-field scaling parameter as well as on the dimension $d\geq 1$. One of the technical difficulties lies in the fact that the key estimate of \cite{oeschlager1987fluctuation} has not been given to the intermediate nonlocal PDE, but only to an asymptotic approximation. Therefore, further justification is provided in this article. The impact of this article is threefold: First, it proves quantitative mean-field estimates for the marginals in the case that the interaction potential $V^\eta$ approximates the Delta distribution - a singular interaction potential which was not included in literature for relative entropy estimates. Second, our main result works not only for a periodic domain (such as the results in \cite{jabin2018quantitative,bresch2019mean}) but also for the whole space. Finally, as a by product we also show that the proposed method of using the relative entropy method can be directly applied for the (both repulsive and attractive) Coulomb interaction case, which is another novelty of this paper.
    
     The first main theorem (Theorem \ref{mainthm}) shows a quantitative propagation of chaos result for any $k$-th marginal of the particle distribution of \eqref{eq:1} in $L^\infty(0,T; L^1(\R^{dk}))$-norm approximating the viscous porous media equation. The moderate interaction parameter is assumed to be scaled algebraically with respect to the number of particles $N$, i.e. $\eta =N^{-\beta/d}$ for some $\beta >0$. The proof is done by connecting the relative entropy method introduced in \cite{jabinWang2016} with the regularized $L^2(\R^d)$-estimate given in \cite{oeschlager1987fluctuation}. Eventually, Theorem \ref{thm.coulomb} shows a quantitative propagation of chaos result in $L^\infty(0,T; L^1(\R^d))$-norm for a particle system with an interaction potential approximating a Coulomb kernel.
    
    The article is organized as follows. In the following section (Section \ref{sec:ideas}) we present our main result (Theorem \ref{mainthm}) and give a guideline on how the relative entropy method can be used in order to derive our main result in the moderate setting. In Section \ref{sec:intermediate} we provide existence results and error estimates concerning the intermediate non-local equation \eqref{eq:3} which converges towards \eqref{eq:2} for $\eta \to 0$. Additionally, we recall error estimates from \cite{oeschlager1987fluctuation} concerning an asymptotic approximation $\rho_{N,L}$ of the solution to the intermediate McKean-Vlasov equation. The relative entropy estimate between the solution to this intermediate problem and that of the moderately interacting particle system \eqref{eq:1} is derived in Lemma \ref{lm.N}, Section \ref{sec:main_proof}. The proof of Theorem \ref{mainthm} is given in the end of this section. An extra discussion in deriving the same result for Coulomb interaction case is given in Section \ref{sec:Coulomb}.

    \section{Main result and ideas}\label{sec:ideas}
    The Kolmogorov forward equation of \eqref{eq:1} reads as
    \begin{equation}\label{eq:li1}
    \partial_t \rho_N^\eta = \frac12 \sum_{i=1}^N \Delta_{x_i} \rho_N^\eta + \sum_{i=1}^N \diver_{x_i}\Big(\rho_N^\eta \frac{1}{2N} \sum_{\substack{j=1}}^N\nabla V^\eta(x_i-x_j)\Big),\quad \rho_N^\eta\big|_{t=0}=\rho_0^{\otimes N},
    \end{equation}
    where $\rho^\eta_N=\rho^\eta_N(t,x_1,\ldots,x_N)$ is the joint law of the particles $(X_{i,t}^{N,\eta})^N_{i=1}$ which solve \eqref{eq:1}. The above equation is satisfied in the sense of distribution by using It\^o's formula.  
    The initial data is normalized in the following sense 
    \begin{align*} 
    \int_{\mathbb{R}^{dN}} \rho_0^{\otimes N} \de x_{1}\ldots \de x_N =\Big(\int_{\mathbb{R}^{d}} \rho_0 \de x\Big)^N=1.
    \end{align*}
    Since for any given $\eta>0$, \eqref{eq:li1} is a linear parabolic equation with smooth coefficients, one directly obtains existence and uniqueness of a classical solution with appropriate assumptions of the initial data $\rho_0$.
    
    We denote with $\rho_{N,k}^\eta$ the $k$-th marginal of $\rho^\eta_N$, for $k=1,2,\ldots,N$, which is given by 
    \begin{align*}
       \rho_{N,k}^\eta  = \int_{\mathbb{R}^{(N-k)d}} \rho_N^\eta(t,x_1,\ldots,x_N) \de x_{k+1}\ldots \de x_N.
    \end{align*}

    Since the key estimate obtained in \cite[Theorem 1]{oeschlager1987fluctuation} will be directly cited in this article, the following assumptions posed in \cite{oeschlager1987fluctuation} are also needed for our main result.
    \begin{assumption}  \label{assV} Concerning the interaction potential $V^\eta = W^\eta \ast W^\eta$ and $W$ given in \eqref{eq:V1} we impose the following assumptions:
    	\begin{enumerate}
    		\item[(i)] \textbf{Scaling:} $\eta=N^{-\beta/d}$ and $\beta\in (0,\frac{d}{2(d+2)})$. 
    		\item[(ii)] \textbf{Probability density function with finite moments:} \\ $W$ is a non-negative radially symmetric function with norm \ \\ $\Vert W \Vert_{L^1(\R^d)}=1$.
    		Additionally, $$\displaystyle\int_{\R^d}|x|^lW(x)dx <\infty,$$ for all $l\in \mathbb{N}$. 
    		\item[(iii)] \textbf{Decay of the Fourier transform:} Let $\mathcal{F}(W)$ denote the Fourier transform of $W$.\\We assume that there exists a constant $C>0$ such that $|\mathcal{F}(W)(\lambda)|\leq C e^{-C|\lambda|}$ for every $\lambda \in \R$ and that for any multi-index $\alpha$ in $\R^d$, there exists a constant $C(\alpha)>0$ such that
    		$$ |D^\alpha_\lambda \mathcal{F}(W)(\lambda)|\leq C(\alpha) (1+|\lambda|^{|\alpha|})|\mathcal{F}(W)(\lambda)|.$$
    		
    	\end{enumerate}	
    \end{assumption}

    Assumption \ref{assV} holds for instance if $W$ is Gaussian, \cite{oeschlager1987fluctuation}.

    \begin{remark}
    The setting of moderately interacting particles for the gradient system \eqref{eq:1} allows $\beta \in (0, \frac{d}{(d+2)})$ if one is purely interested in the mean-field limit of the empirical measure, see \cite{oelschlager1985law}. We use the more restrictive scaling $\beta \in (0, \frac{d}{2(d+2)})$ in order to derive a quantitative estimate for the marginals of the particle distribution of \eqref{eq:1}.
    \end{remark}
    
    \begin{assumption}\label{assInit}For $m\in \N$, let
    		$$C^{m}_{1,\infty}:=\{\rho~| D^\alpha\rho\in L^1(\R^d)\cap L^\infty(\R^d), \forall |\alpha|\leq m\}.$$ The initial data fulfils the following assumption:  Let $(\xi_i)_{i=1}^N $ be given i.i.d. random variables with common law $\mu_0$ and probability density $\rho_0 \in C^{\infty}_{1,\infty}(\R^d)$.
    \end{assumption}
    We need Assumption \ref{assInit} in order to apply results obtained in \cite{oeschlager1987fluctuation} (especially Lemma \ref{lemL2}) under strong regularity conditions on $\rho_0$. Estimates obtained in Section \ref{sec:intermediate} on the PDE level will hold under weaker assumptions on the initial data.
    
    The main result of this article is the following quantitative propagation of chaos result for any $k$-th marginal $\rho^\eta_{N,k}$ of the joint law $\rho_{N}^\eta$ of the particles fulfilling \eqref{eq:1}  in $L^\infty(0,T;L^1(\R^{dk}))$-norm.
    \begin{theorem}[Quantitative propagation of chaos result in $L^1(\R^{dk})$--norm]\label{mainthm}
        Let Assumptions \ref{assV} and \ref{assInit} hold. Then there exist  $N_0\in\mathbb{N}$ so that for all $ N\geq N_0$ and $\varepsilon\in (0, 1/4-\beta(d+2)/2d)$ it holds that 
        \begin{align}\label{eq:result}
            \|\rho_{N,k}^\eta - \rb^{\otimes k}\|_{L^\infty(0,T;L^1(\mathbb{R}^{dk}))} \le \frac{C k^{1/2}}{N^{\min\{\frac14+\varepsilon,\frac{2\beta}{d}\}}} \quad \forall k>0,
        \end{align}
    where $C>0$ is a constant independent of $N$ and $k$.
    \end{theorem}
    
    \subsection{Discussion of the convergence rate} Since $0<\eps<1/4-\beta(d+2)/2d < 1/4$ in Theorem \ref{mainthm}, the convergence rate of \eqref{eq:result} is smaller than $1/2$. For dimension $d=1$, a simple calculation shows that the maximal rate of convergence lies between $1/4$ and $1/3$ and becomes arbitrarily small if $\beta \sim 0$. For $d\geq 2$, it holds that $\min\{\frac{1}{4}+ \eps, \frac{2 \beta}{d}\} =  \frac{2\beta}{d}$since for $d\geq 2$
    $$\frac{2 \beta}{d} \leq \frac{1}{4}~~~\text{ due to the choice of }\beta \in \Big(0, \frac{d}{2(d+2)}\Big).$$ It shows that for $d\geq 2$ the convergence rate lies between $(0,\frac{1}{d+2})$ and becomes arbitrary slow for $\beta \sim 0$. This corresponds to the observation made e.g. in \cite{oeschlager1987fluctuation}, that due to the change of the regime from \textit{moderately} towards \textit{weakly interacting particles} for $\beta\to 0$ (which leads to a non-local PDE in the limit), the mean-field convergence towards the limiting local PDE \eqref{eq:2} is very slow for small values of $\beta>0$. 
    
    The proof of Theorem \ref{mainthm} is done by exploiting an asymptotic approximating sequence $\rho_{N,L}$ of the intermediate PDE problem introduced in \cite{oeschlager1987fluctuation} and implies that the convergence rate of any $k$-th marginal $\rho^\eta_{N,k}$ towards $\rho_{N,L}^{\otimes k}$ in $L^\infty(0,T; L^1(\R^{dk}))$ is smaller or equal $1/4+\eps$, i.e.
    \begin{align*}
     \|\rho_{N,k}^\eta - (\rho_{N,L})^{\otimes k}\|_{L^\infty(0,T;L^1(\mathbb{R}^{dk}))} \le CN^{-1/4-\eps}.
    \end{align*} The minimum condition in \eqref{eq:result} in the convergence rate originates from the PDE error estimates between the approximated solution $\rho_{N,L}$ and the PDE solution $\rb$ to \eqref{eq:2} in $H^s(\R^d)$-norm, i.e. form the fact that
    \begin{align*}
    \Vert \rb - \rho_{N,L} \Vert_{L^\infty(0,T; H^s(\R^d))} \leq CN^{-2\beta/ d} (= C\eta^2),
    \end{align*}
    see Lemma \ref{lem.E} in Section \ref{sec:intermediate}.

    \begin{remark}[Propagation of chaos]
    We remark that Theorem \ref{mainthm} implies a quantitative propagation of chaos result in the \textit{classical (weak) sense}, i.e. that for every $t\in(0,T)$ the joint law $\rho_{N}^\eta(t)$ of the particle system is $\rb(t)$-chaotic. We refer to \cite{Sznitman91} and \cite{HaurayMischler14} for the exact definition of propagation of chaos, chaoticity and their implications.
    \end{remark}
    \begin{remark}[Convergence on the whole space $\R^d$]
    In the seminal work \cite{jabin2018quantitative}, the authors were able to derive a quantitative convergence result by relative entropy estimates for interaction kernels of $W^{-1,\infty}$-type on the torus in the weak interaction regime. The result requires that the solution to the limiting PDE is a strictly positive function - an assumption which is not trivially fulfilled if the domain is unbounded (see \cite[Remark 3]{jabin2018quantitative}). In this article, applying an estimate in $L^2(\R^d)$-norm for the smoothed empirical measure from \cite{oeschlager1987fluctuation} allows us to provide a result which holds on the whole space and does not require positivity of solutions to \eqref{eq:2}. 
    \end{remark}
    
    \subsection{Main idea of the proof of Theorem \ref{mainthm}} Before we state the main idea of our main theorem, we need to recall some definitions concerning the relative entropy. 
    For $m \in \N$ and two probability density functions $p,q$ on $\R^{dm}$, we define the \textit{(non-normalized) relative entropy} between $p$ and $q$ as
    \begin{align*}
    &\mathcal{H}_m(p|q) :=\int_{\mathbb{R}^{dm}}p\log \frac{p}{q} \de x_1\ldots \de x_m,
    \end{align*}
    and by 
    \begin{align}\label{total_relative_entropy_def}
    \mathcal{H}(\rho_{N}^\eta| q_N ) (t):=\frac{1}{N} \mathcal{H}_N(\rho_{N}^\eta|q_N ) (t) = \frac{1}{N}\int_{\mathbb{R}^{dN}}\rho_{N}^\eta(t)\log \frac{\rho_{N}^\eta(t)}{q_N(t)} \de x_1\ldots \de x_N,
    \end{align}
    we denote the \textit{total relative entropy} of the joint law of particle system \eqref{eq:1} with respect to any probability density function $q_N(t)$ on $\R^{dN}$ at time $t\geq 0$.
    
    The proof of Theorem \ref{mainthm} is given by a combination of the relative entropy method and an estimate derived in \cite{oeschlager1987fluctuation} in the framework of fluctuation theory. Following the strategy introduced in \cite{jabin2018quantitative}, the Csisz\'ar-Kullback-Pinsker inequality (see Lemma \ref{lem:ckp} in the appendix) as well as the fact that the relative entropy of the $k$-th marginal can be bounded by the total relative entropy (see Lemma \ref{lem:superadd} in the appendix) shows that
    \begin{align}\label{idea:CKP}
    \sup_{0<t<T}\|\rho_{N,k}^\eta(t) - \rb^{\otimes k}(t)\|_{L^1(\mathbb{R}^{dk})}^2 &\le \sup_{0<t<T}2 \mathcal{H}_k(\rho_{N,k}^\eta|\rb^{\otimes k} ) (t) \nonumber \\ &\leq\sup_{0<t<T}  \frac{4k}{N} \mathcal{H}_N(\rho_{N}^\eta|\rb^{\otimes N} ) (t).
    \end{align}
    Inequality \eqref{idea:CKP} indicates that our strategy is to investigate quantitative estimates of the total relative entropy between the joint law $\rho_N^\eta$ of the coupled particle system \eqref{eq:1} and the factorized law $\rb^{\otimes N}$ in order to derive the desired result.  \\
    
    The estimate for the relative entropy between $\rho_N^\eta$ and $\rb^{\otimes N}$ is done in two steps: For both steps, we need to introduce an intermediate particle system of McKean-Vlasov type: For fixed $\eta >0$, let $(\bar X_{i,t}^{\eta})_{i=1}^N$ be the solution to the following (uncoupled) system of SDEs
     \begin{eqnarray}
    \label{eq:interPar}
    &&\de\bar X_{i,t}^{\eta} =- \frac12 \nabla V^\eta*\rb^\eta(\bar X_{i,t}^{\eta}) \de t +  \de B_t^i, \\
    \nonumber && \bar X_{i,0}^{\eta}  = \xi_i,\quad i=1,\ldots,N,
    \end{eqnarray}
    where $\rb^\eta(\cdot,t)$ is the density function of the i.i.d. random processes \newline $\bar X_{1,t}^{\eta},\ldots, \bar X_{N,t}^{\eta}$, and $(\xi_i)_{i=1}^N$ is the same initial condition as in \eqref{eq:1}. By using It\^o's formula, the density function $\rb^\eta$ fulfils the so-called intermediate nonlocal problem, namely
    \begin{equation}\label{eq:3}
    \partial_t \rb^\eta = \frac12 \Delta \rb^\eta + \frac12 \diver(\rb^\eta \nabla V^\eta * \rb^\eta), \qquad
    \rb^\eta (0) = \rho_0,
    \end{equation}
    in the very weak sense for any initial condition $\rho_0 \in H^s(\R^d)$, see e.g. \cite[Lemma 10]{chen2019rigorous} for a rigorous justification.
    
    We split our proof in estimates between the joint law $\rb_N^\eta: =(\rb^\eta)^{\otimes N}$ of the intermediate system \eqref{eq:interPar} and $(\rb)^{\otimes N}$ as well as between the joint law $\rho^\eta_{N}$ of the interacting particle system \eqref{eq:1} and $ (\rb^\eta)^{\otimes N}$. 
    \begin{enumerate}
    \item First, we show that the solution $\rb^\eta$ to \eqref{eq:3} and the solution $\rb$ to \eqref{eq:1} satisfy the relative entropy estimate
    \begin{align*}
    \mathcal{H}_1(\rb^\eta|\rb) \leq C\eta^4 = CN^{-4\beta/d}.
    \end{align*}
    Since particles $\bar{X}_{i,t}^\eta(t)$ associated with \eqref{eq:3}  as well as particles associated with \eqref{eq:2} are already independent, their joint distribution is a factorized function. Hence, Lemma \ref{lem:superadd} as well as the fact that the appearing probability distributions are factorized shows
    \begin{align*}
    \mathcal{H}_k(\rb^\eta_{N,k}| \rb^{\otimes k }) = \mathcal{H}_k((\rb^\eta)^{\otimes k}| \rb^{\otimes k }) \leq k \mathcal{H}(\rb^\eta_{N}| \rb^{\otimes N }) &= k\mathcal{H}_1(\rb^\eta|\rb)\\
    &\leq CkN^{-4\beta/d},
    \end{align*}
    where $\rb^\eta_{N,k}$ denotes the $k$-th marginal of $\rb^\eta_N$, see Section \ref{sec:intermediate} for details.
    
    \item 
    In the second step, we show a relative entropy estimate between the joint law $\rho_N^\eta$ and the chaotic law $\rb_N^\eta: = (\rb^\eta)^{\otimes N}$, where we remark that $\rb^\eta$ still depends on $\eta>0$ and hence on the number of particles $N$.
    
    \noindent Let for any time $t>0 $ \begin{align}\label{emp_m}
    \mu_N(t) := \frac{1}{N} \sum_{i=1}^N \delta_{X_{i,t}^{N, \eta}(t)} 
    \end{align} denote the empirical measure with respect to particle system \eqref{eq:2}. We show that the total relative entropy between $\rho_{N}^\eta$ and the intermediate joint law $\rb^\eta_N$ fulfils the following inequality
    \begin{align}\label{idea.key_estimate}
    	\hspace{0.9cm}&\mathcal{H}(\rho^\eta_N|\rb^\eta_N)(t)+\frac{1}{4N} \int^t_0\int_{\R^{dN}}\sum_{i=1}^N \left|\nabla_{x_i} \log \frac{\rho^\eta_N}{\rb^\eta_N}\right|^2 \rho^\eta_N \de x_1\ldots \de x_N \de s \nonumber\\
    	&\leq  \frac{1}{2}\int_{0}^t\mathbb{E}\Big(\Big\langle \mu_N, \big|\nabla V^\eta \ast (\mu_N - \rho_{N,L})\big|^2 \Big\rangle\Big) \de s \nonumber \\
    	&+ \frac{1}{2}\int_{0}^t\mathbb{E}\Big(\Big\langle \mu_N, \big|\nabla V^\eta \ast (\rho_{N,L} - \rb^\eta)\big|^2 \Big\rangle\Big) \de s,
    \end{align}
    where $\rho_{N,L}$ denotes an approximating sequence introduced in \cite{oeschlager1987fluctuation}, see Lemma \ref{lemmassymp}.
    An application of \cite{oeschlager1987fluctuation} (see Lemma \ref{lemL2}) shows that for small $\eps>0$
    \begin{align}\label{idea.estimate_pNL}
    \int_{0}^t\mathbb{E}\Big(\Big\langle \mu_N, \big|\nabla V^\eta \ast (\mu_N - \rho_{N,L})\big|^2 \Big\rangle\Big) \de s \leq CN^{-1/2-\eps},
    \end{align}
    which by exploiting additional error estimates between $\rho_{N,L}$ and $\rb_N^\eta$ leads to
    \begin{align*}
    \mathcal{H}(\rho^\eta_N|\rb^\eta_N)(t) \leq C(N^{-1/2-\eps} + N^{-4\beta/d}) \leq \frac C {N^{\min\{1/2 + \eps, 4\beta/d \}}}.
    \end{align*}
    We remark that \eqref{idea.key_estimate} is the key estimate of this article and it shows clearly that by applying the result in \cite{oeschlager1987fluctuation} (which was shown in order to derive a fluctuation result) one does not need to prove the technically challenging large deviation principle of \cite{jabin2017mean,jabinWang2016}. Finally, with \eqref{idea:CKP}, we find that
    \begin{align*}
    \sup_{0<t<T}\|\rho_{N,k}^\eta(t) - (\rb^\eta)^{\otimes k}(t)\|_{L^1(\mathbb{R}^{dk})} &\leq\sup_{0<t<T} Ck^{1/2} \sqrt{\mathcal{H}(\rho^\eta_N|\rb^\eta_N)(t)}\\&\leq \frac {Ck^{1/2}} {N^{\min\{1/4 + \eps, 2\beta/d \}}}.
    \end{align*}For details we refer to reader to Section \ref{sec:main_proof}.

    \end{enumerate}
    
    Triangle inequality and the relative entropy estimates derived in Step 1 and 2, lead to \eqref{eq:result}, which finishes the proof.
    \begin{remark}
    Inequality \eqref{idea.estimate_pNL} suggests that the convergence rate in relative entropy between $\rho_{N}^\eta$ and $\rho_{N,L}$ is of order $N^{-1/2- \eps}$, which fits the observation of Oelschl\"ager in \cite{oeschlager1987fluctuation} where a fluctuation theorem is shown for $\sqrt{N}(\mu_N - \rho_{N,L})$. 
    \end{remark}

    \section{Intermediate nonlocal problem and PDE error estimates} \label{sec:intermediate}
    In this section we prove a well-posedness result and error estimates concerning equation \eqref{eq:3} which is an nonlocal approximation of \eqref{eq:2} for $\eta \to 0$. We provide a convergence result in a suitable Bochner space as well as with respect to the relative entropy $\mathcal{H}_1$. Additionally, an asymptotic approximation $\rho_{N,L}$ of the nonlocal problem \eqref{eq:3} obtained in \cite{oeschlager1987fluctuation} is recalled in this section. Furthermore we will give concrete estimates between the solution of \eqref{eq:3} and the solution $\rho_{N,L}$ to the asymptotic equation. 
    
    In the beginning of this section, for the reader's convenience we recall equation \eqref{eq:3}, which is an intermediate nonlocal equation of mean-field type 
    \begin{equation*}
    \partial_t \rb^\eta = \frac12 \Delta \rb^\eta + \frac12 \diver(\rb^\eta \nabla V^\eta * \rb^\eta), \qquad
    \rb^\eta (0) = \rho_0.
    \end{equation*}
    It can be interpreted as mean-field limit of particle system \eqref{eq:1} for fixed $\eta>0$ in the weak interaction regime.
    This nonlocal diffusion equation cannot be directly solved by classical theory for parabolic partial differential equations. We remark that well-posedness of equation \eqref{eq:3} has been studied in \cite{oelschlager2001viscous}, however, the results can not be directly applied for the framework of the relative entropy method. Therefore, for completeness we provide suitable estimates in this section. \\
    
    First, we start with the following lemma from \cite[Lemma 1]{oeschlager1987fluctuation} which gives an asymptotic approximation of the above intermediate problem as well as well-posedness for the viscous porous media equation:
    \begin{lemma}[Well-posedness of \eqref{eq:2} and asymptotic estimate, \cite{oeschlager1987fluctuation}]\label{lemmassymp}
    	Let Assumption \ref{assInit} holds, then for any $T>0$, \eqref{eq:2} has a unique solution $\rho\in C^{\infty}_{1,\infty}(\R^d\times [0,T])$ and there exists $\rho_l, R_{N,l}\in C^\infty_{1,\infty}(\R^d\times [0,T])$, $l=1,\cdots,L$ with
    	\begin{equation}
    	\sup_{N\in\mathbb{N}}(\|R_{N,l}\|_{L^1(\R^d\times [0,T])}+\|R_{N,l}\|_{L^\infty(\R^d\times [0,T])})\leq C(l)\eta^{2l+2}=\frac{C(l)}{N^{2\beta (l+1)/d}} 
    	\end{equation}
    	such that 
    	\begin{equation}\label{RNL}
    	\rho_{N,L}=\rb+ \displaystyle\sum^L_{l=1}\eta^{2l} \rho_l=\rb + \displaystyle\sum^L_{l=1}N^{-2\beta l/d} \rho_l
    	\end{equation}
    	solves
    	\begin{equation}\label{eq:L}
    	\partial_t \rho_{N,L} = \frac12 \Delta \rho_{N,L} + \frac12 \diver(\rho_{N,L} \nabla V^\eta * \rho_{N,L}) + R_{N,L}, \qquad
    	\rho_{N,L} (0) = \rho_0.
    	\end{equation}
    \end{lemma}
    For a proof of Lemma \ref{lemmassymp} we refer to \cite{oeschlager1987fluctuation}. The functions $\rho_l$ are constructed as solutions to the linearized version of \eqref{eq:2} plus a perturbation which depends on the functions $\rb, \rho_1, \ldots, \rho_{l-1}$, we refer to \cite[Appendix A]{oelschlager2001viscous} for a formal derivation of the expansion.  \\
    
    The main result of this section is an existence and uniqueness result for smooth solutions to problem \eqref{eq:3} for small $\eta >0$ and estimates for the difference $\rb^\eta-\rho_{N,L}$:
    \begin{lemma}\label{lem.E}
    	Let $0\leq\rho_0\in H^s(\R^d)$ for $s>d/2+2$. Then for any $T>0$ there exists $\eta_0>0$ so that for $0<\eta<\eta_0$ problem \eqref{eq:3} has a unique non-negative solution  $\rb^\eta \in L^\infty(0,T;H^s(\R^d)) \cap L^2(0,T;H^{s+1}(\R^d))$. Furthermore, the following uniform estimates hold
    	\begin{equation}\label{error_rb_rb_eta}
    	\|\rb-\rb^\eta\|_{L^\infty(0,T;H^s(\R^d))}+\|\rb-\rb^\eta\|_{L^2(0,T;H^{s+1}(\R^d))}\leq C\eta^2,
    	\end{equation}
    	and
    		\begin{align}
    		\label{eq:reEn}
    		\sup_{0\leq t\leq T}\mathcal{H}_1(\rb^\eta|\rb)=\sup_{0\leq t\leq T}\int_{\R^d}\rb^\eta\log\frac{\rb^\eta}{\rb}\de x\leq C\eta^4,
    	\end{align}
    where $C>0$ depends on $T$, $d$, and $\rb$.
    
    Additionally, under Assumption \ref{assInit}, the following estimate holds
    	\begin{equation}\label{eq:error_approx}
    	\|\rb^\eta-\rho_{N,L}\|_{L^\infty(0,T;H^s(\R^d))}+\|\rb^\eta-\rho_{N,L}\|_{L^2(0,T;H^{s+1}(\R^d))}\leq C\eta^2.
    	\end{equation}
    
    \end{lemma}
    
    \begin{proof} 
    	By applying Banach fixed point theorem, it is standard to obtain for each fixed $\eta>0$ the local existence and uniqueness of a non-negative solution $\rb^\eta$ to problem \eqref{eq:3} until a time point $t^*(\eta)>0$, see for instance the techniques in \cite[Lemma 4]{chen2019rigorous} for a more general setting which includes \eqref{eq:3}.
    	
    	 We present here only the uniform estimates in $\eta$ for $w^\eta=\rb-\rb^\eta$, which guarantee that \eqref{error_rb_rb_eta} holds and
    that the solution $\rb^\eta$ can be extended to any fixed time $T>0$. For simplicity of the notation, we drop the index $\eta$ of $w^\eta$ from now on.
    
    Given that the solution $\rb^\eta$ exists on a time interval $[0,t^*(\eta)]$, we first note that the difference $w=\rb-\rb^\eta$ satisfies the equation
    	\begin{eqnarray}
    		\label{eq:w}\partial_t w=\frac12 \Delta w+\frac12 \diver (w\nabla\rb+(\rb-w)\nabla(\rb-V^\eta*\rb)+(\rb-w)\nabla V^\eta*w),
    	\end{eqnarray}
     on $(0,t^*(\eta))$ with initial data $w|_{t=0}=0$.
    		
    	For any multi-index $\alpha\in \N^d$ with $|\alpha|\leq s$, we apply $D^\alpha$ to the above equation, take the inner product with the test function $D^\alpha w$
    	and obtain
    	\begin{align}\label{eq:change_Hs_w}
    		&\frac{\de }{\de t}\int_{\R^d} |D^\alpha w|^2 \de x +\int_{\R^d} |D^\alpha\nabla w|^2 \de x=- \int_{\R^d} D^\alpha (w\nabla\rb)D^\alpha \nabla w \de x \nonumber\\ 
    		&-\int_{\R^d} D^\alpha \big((\rb-w)\nabla(\rb-V^\eta*\rb)\big)D^\alpha \nabla w \de x \nonumber \\
    		&-\int_{\R^d} D^\alpha \big((\rb-w)\nabla V^\eta*w\big)D^\alpha \nabla w \de x.
    	\end{align}
    The three terms on the right hand side, which will be denoted by $I_1$, $I_2$, and $I_3$ in chronological order, can be estimated separately.
    \begin{align*}
    I_1(t) &:= - \int_{\R^d} D^\alpha (w(t)\nabla\rb(t))D^\alpha \nabla w(t) \de x \\
    I_2(t) &:= -\int_{\R^d} D^\alpha \big((\rb-w)(t)\nabla(\rb(t)-V^\eta*\rb(t))\big)D^\alpha \nabla w(t) \de x\\
    I_3(t)&:=-\int_{\R^d} D^\alpha \big((\rb-w)(t)\nabla V^\eta*w(t)\big)D^\alpha \nabla w (t)\de x.
    \end{align*}
     For simplicity, we drop the dependency on $t>0$ in the following lines. Term $I_1$ can be easily controlled in the following way:
    \begin{align}\label{ineq:I_1}
    |I_1| &\leq \| D^\alpha (w\nabla\rb)\|_{L^2(\R^d)} \|D^\alpha\nabla w\|_{L^2(\R^d)}\nonumber\\
    &\leq C \|w\|_{H^s(\R^d)}\|\nabla\rb\|_{H^{s}(\R^d)} \|D^\alpha\nabla w\|_{L^2(\R^d)},
    \end{align}
    where we used the Moser-type inequality for the product (Lemma \ref{lem.moser}) as well as $H^s(\R^d) \hookrightarrow L^\infty(\R^d)$ in the last step. For $I_2$, we need to use the fact that $\rb$, which is given by Lemma \ref{lemmassymp}, is a smooth solution of \eqref{eq:1}, and the second moment of potential $V^\eta$ is of order $\eta^2$. Thus, by H{\"o}lder's inequality and Taylor's expansion
    \begin{align}\label{ineq:I_2}
    &|I_2|\leq  C \|\rb-w\|_{H^s(\R^d)}\sum_{|\beta|\leq s+1}\|D^\beta(\rb-V^\eta*\rb)\|_{L^{\infty}(\R^d)} \|D^\alpha\nabla w\|_{L^2(\R^d)}\nonumber\\
    &=C \|\rb-w\|_{H^s(\R^d)}\sum_{|\beta|\leq s+1} \Big\|\int_{\R^d}V^\eta(y)(D^{\beta}\rb(\cdot-y)-D^{\beta}\rb(\cdot))\de y\Big\|_{L^\infty(\R^d)} \|D^\alpha\nabla w\|_{L^2(\R^d)}\nonumber\\
    	&=C \|\rb-w\|_{H^s(\R^d)}\sum_{|\beta|\leq s+1} \Big\|D^2
    	D^{\beta}\rb\Big\|_{L^\infty(\R^d)}\int_{\R^d}V^\eta(y)|y|^2
    \de y \|D^\alpha\nabla w\|_{L^2(\R^d)}\nonumber\\
    	&\leq C \|\rb-w\|_{H^s(\R^d)}\|\rb\|_{W^{s+3,\infty}(\R^d)}\eta^2
    	\|D^\alpha\nabla w\|_{L^2(\R^d)}\nonumber\\
    	&\leq  C\eta^2	(1+\|w\|_{H^s(\R^d)})\|D^\alpha\nabla w\|_{L^2(\R^d)},
    \end{align}
    	where $C>0$ depends on $\|\rb\|_{L^\infty(0,T;H^{s+3+d/2}(\R^d))}$. Here we have used Taylor's expansion in the third line as well as the fact that $V^\eta$ is symmetric, so that $\int_{\R^d}y V^\eta(y)\de y=0.$
    	
    For $I_3$ notice that $\rb-w=\rb^\eta\geq 0$. By using the convolution structure of $V^\eta=W^\eta*W^\eta$ and recalling the commutator notation $[D^\alpha,f]g=D^\alpha(fg)-fD^\alpha g$, we split the estimate for $I_3$ into estimates for two further terms $I_{31}, I_{32}$ in the following way, 
    \begin{align*}
    	I_3&=-\int_{\R^d} D^\alpha \big((\rb-w)\nabla V^\eta*w\big)D^\alpha \nabla w \de x\\
    	&= -\int_{\R^d} \big([D^\alpha,\rb-w] \nabla V^\eta*w\big) D^\alpha\nabla w \de x -\int_{\R^d} (\rb-w) D^\alpha \nabla V^\eta*w  D^\alpha \nabla w \de x\\
    	&= -\int_{\R^d} \big([D^\alpha,\rb-w] \nabla V^\eta*w\big) D^\alpha\nabla w \de x-\int_{\R^d} (\rb-w) |D^\alpha \nabla W^\eta*w|^2 \de x  \\
    	& - \int_{\R^d} D^\alpha \nabla W^\eta*w \Big( W^\eta * ((\rb-w) D^\alpha\nabla w)-  (\rb-w) W^\eta * D^\alpha\nabla w\Big) \de x\\
    	&=I_{31}-\int_{\R^d} \rb_\eta |D^\alpha \nabla W^\eta*w|^2 \de x +I_{32},
    \end{align*}
    where we used that $\int_{\R^{d}} (W^\eta \ast f) g \de x = \int_{\R^{d}} f (W^\eta \ast g) \de x$ due to symmetry of $W^\eta$ and
    \begin{align*}
    I_{31}  &:=- \int_{\R^d} \big([D^\alpha,\rb-w] \nabla V^\eta*w \big) D^\alpha\nabla w \de x\\
    I_{32}&:=- \int_{\R^d} D^\alpha \nabla W^\eta*w \Big( W^\eta * ((\rb-w) D^\alpha\nabla w)-  (\rb-w) W^\eta * D^\alpha\nabla w\Big) \de x.
    \end{align*}
    $I_{31}$ can be easily estimated by using the Moser--type inequality for the commutator (Lemma \ref{lem.comm} in the appendix with $|\alpha|\leq s$) and Young's convolution inequality together with $\Vert V^\eta \Vert_{L^1(\R^d)} =1$, i.e.
    \begin{align}\label{ineq:I_31}
    	I_{31} 
    	&\leq  C\Big( \|D ( \rb-w)\|_{L^\infty(\R^d)}\|D^{s}V^\eta*w\|_{L^2(\R^d)}\nonumber\\
    	&\hspace{0.5cm}+  \|D^s (\rb-w)\|_{L^2(\R^d)}\|\nabla V^\eta*w\|_{L^\infty(\R^d)}\Big) \|D^\alpha\nabla w\|_{L^2(\R^d)} \nonumber\\
    	&\leq   C(1+\|w\|_{H^s(\R^d)})\|w\|_{H^s(\R^d)} \|D^\alpha\nabla w\|_{L^2(\R^d)},
    \end{align}
    where $C>0$ depends on the given solution $\rb$ of \eqref{eq:1} and we have used that $H^s(\R^d) \hookrightarrow W^{1,\infty}(\R^d)$.

    By using the mean-value theorem as well as Young's convolution inequality we obtain for $I_{32}$
    \begin{align}\label{ineq:I_32}
    	I_{32}
     &= \int_{\R^d} D^\alpha \nabla W^\eta*w (x) \int_{\R^d}W^\eta(x-y)\big((\rb-w)(x)-(\rb-w)(y)\big) D^\alpha\nabla w(y)\de y \de x \nonumber\\
     &\leq \|\nabla(\rb-w)\|_{L^\infty(\R^d)} \int_{\R^d} | D^\alpha \nabla W^\eta*w (x) | \int_{\R^d}W^\eta(x-y) |x-y| |D^\alpha\nabla w(y)|\de y \de x \nonumber\\
     &= \|\nabla(\rb-w)\|_{L^\infty(\R^d)} \int_{\R^d} | D^\alpha \nabla W^\eta*w (x) | \Big|W^\eta(\cdot) |\cdot|  \ast |D^\alpha\nabla w|\Big|(x) \de x\nonumber\\
    &\leq  \|\nabla(\rb-w)\|_{L^\infty(\R^d)} \big\||\cdot|W^\eta(\cdot)\big\|_{L^1(\R^d)} \| D^\alpha \nabla W^\eta*w\|_{L^2(\R^d)} \| D^\alpha \nabla w\|_{L^2(\R^d)}\nonumber\\
    &\leq C\eta (1+ \|w\|_{H^s(\R^d)}) \int_{\R^d} |D^\alpha\nabla w|^2 \de x
    \end{align}
    where $C$ depends on $\rb$ as well as on the Sobolev embedding constant for $H^s(\R^d)\hookrightarrow W^{1,\infty}(\R^d)$, and we have used the fact that $\big\||\cdot|W^\eta(\cdot)\big\|_{L^1(\R^d)}\leq \eta \big\||\cdot|W(\cdot)\big\|_{L^1(\R^d)}$,  and $\Vert D^\alpha \nabla W^\eta \ast w \Vert_{L^2(\R^d)} \leq \Vert D^\alpha \nabla w \Vert_{L^2(\R^d)}$.
    
    Combining estimates \eqref{eq:change_Hs_w}, \eqref{ineq:I_1}, \eqref{ineq:I_2} and \eqref{ineq:I_32}, we obtain by using \begin{align*}
    \Vert D^\alpha \nabla w \Vert_{L^2(\R^d)}\Vert w \Vert_{H^s(\R^d)}\leq 1/4 \Vert D^\alpha \nabla w \Vert_{L^2(\R^d)}^2 + C \Vert w \Vert_{H^s(\R^d)}^2
    \end{align*} from Young's inequality that
    \begin{align*}
    &\frac{\de }{\de t}\int_{\R^d} |D^\alpha w|^2 \de x +\int_{\R^d} |D^\alpha\nabla w|^2 \de x \\
    &\leq C \Vert D^\alpha \nabla w \Vert_{L^2(\R^d)} \Big(\Vert w \Vert_{H^{s}(\R^d)}(1+\eta^2) + \Vert w \Vert_{H^{s}(\R^d)}^2 + \eta^2 \Big)\\
    & + C \Vert D^\alpha \nabla w \Vert_{L^2(\R^d)}^2(\eta + \eta \Vert w \Vert_{H^{s}(\R^d)})\\
    &\leq C \Vert D^\alpha \nabla w \Vert_{L^2(\R^d)} \Big(\eta^2\Vert w \Vert_{H^{s}(\R^d)}+ \Vert w \Vert_{H^{s}(\R^d)}^2 + \eta^2 \Big)\\
    & +  \Vert D^\alpha \nabla w \Vert_{L^2(\R^d)}^2\Big(\frac{1}{4}+\eta C+ \eta C\Vert w \Vert_{H^{s}(\R^d)}\Big)+C\Vert w \Vert_{H^{s}(\R^d)}^2,
    \end{align*}
    which holds on $(0,t^*(\eta))$ and where the constant $C>0$ changes value from line to line. Finally, assuming that $\eta < \min\{1/(4C) ,1\}$, summing up the multi-indexes $|\alpha|\leq s$, 

    we obtain by Young's inequality and in particular $$\sum_{|\alpha|\leq s}\Vert D^\alpha \nabla w \Vert_{L^2(\R^d)} \Vert w \Vert_{H^{s}(\R^d)}^2 \leq \frac{1}{4}\Vert w \Vert_{H^{s+1}(\R^d)}^2 \Vert w \Vert_{H^{s}(\R^d)}^2 + C_1\Vert w \Vert_{H^{s}(\R^d)}^2, $$ that 
    \begin{align}\label{ineq:w_on_0_tstar}
    	&\dfrac{\de }{\de t} \|w\|^2_{H^s(\R^d)} +\Big(\frac14-\frac{1}{2}\|w\|^2_{H^s(\R^d)} - \frac{1}{4} \Vert w \Vert_{H^{s}(\R^d)} \Big)\|w\|^2_{H^{s+1}(\R^d)}
    	\leq C_2\|w\|^2_{H^s(\R^d)}+C_3\eta^4,
    \end{align}	
    with $C_2, C_3 >0$ not depending on $\eta>0$. Since the initial data satisfies $\|w(\cdot,0)\|^2_{H^s(\R^d)}=0$, this allows us to conclude by standard techniques that $\exists\eta_0>0$ such that for $\eta<\eta_0$ it holds
    \begin{align}\label{ineq:uniform_estimates}
    \|w\|_{L^\infty(0,T;H^s(\R^d))}+\|w\|_{L^2(0,T;H^{s+1}(\R^d))}\leq C\eta^2,
    \end{align}
    where $C$ depends on $T$, $d$, and $\rb$. 
    
    For the reader's convenience -- since the dependence of $w$ and $t^*(\eta)\leq  T$ on $\eta>0$ requires some care -- we present the proof of \eqref{ineq:uniform_estimates}:
    Let us fix $\eta < \eta_0$, where $\eta_0$ will be chosen later. Since $\|w(\cdot,0)\|^2_{H^s(\R^d)}=0$, by potentially reducing the value of $t^*(\eta)$ to $t_1^\star(\eta)$, we can conclude by the regularity of $w$ that $\sup_{0<t<t_1^*(\eta)}\Vert w \Vert_{H^s(\R^d)} <1/2$. Hence, by Gronwall's inequality we get from \eqref{ineq:w_on_0_tstar} that
    \begin{align}\label{ineq:first_iteration}
    \sup_{0<t<t_1^*(\eta)}\Vert w \Vert_{H^s(\R^d)}^2 \leq e^{C_2 t_1^*(\eta)}(\Vert w(0) \Vert_{H^s(\R^d)}^2 + t_1^*(\eta)C_3 \eta^4) \leq e^{C_2 T}  TC_3 \eta^4 
    \end{align}
    and consequently 
    \begin{align*}
    \|w\|_{L^\infty(0,t_1^*(\eta);H^s(\R^d))}+\|w\|_{L^2(0,t_1^*(\eta);H^{s+1}(\R^d))}\leq C(T)\eta^2.
    \end{align*}
    If $t_1^*(\eta) =T$ we are done, else taking $w(t_1^*(\eta))$ as new starting point, we notice that by setting $\eta_0^4< 1/(2C_3 Te^{C_2T})$ we get the existence of a time point $t_1^*(\eta)< t_2^*(\eta)\leq T$ such that $\sup_{t_1^*(\eta)<t<t_2^*(\eta)}\Vert w \Vert_{H^s(\R^d)} <1/2$.
    Repeating the argument from above, we know that \eqref{ineq:first_iteration} holds on $(0, t_2^*(\eta))$. This can be done as long as $t_k^*(\eta)\leq T$ since then the condition $e^{C_2 T}  TC_3 \eta^4 < 1/2$ is always fulfilled. By a contradiction argument, we see that we can actually reach the time point $T>0$ under the condition that $\eta_0^4< 1/(2C_3 Te^{C_2T})$, hence \eqref{ineq:uniform_estimates} holds. 
    
    We will use the uniform estimate \eqref{ineq:uniform_estimates} in order to obtain the relative entropy estimate in \eqref{eq:reEn}. Calculating the time derivative of the relative entropy by using both equations \eqref{eq:1} and \eqref{eq:3} leads to the following equation
    \begin{align*}
    	&\frac{\de }{\de t} \mathcal{H}_1(\rb^\eta|\rb)= \int_{\R^d} \partial_t\rb^\eta \log\bigg(\frac{\rb^\eta}{\rb}\bigg) \de x + \int_{\R^d} \Big(\pa_t \rb^\eta -\frac{\rb^\eta \pa_t \rb}{\rb}\Big)\de x \\
    	&= \frac{1}{2}\int_{\R^d}\big( \Delta \rb^\eta + \text{div}\big(\rb^\eta \nabla V^\eta \ast \rb^\eta\big)\big)\log\bigg(\frac{\rb^\eta}{\rb}\bigg) \de x - \frac{1}{2}\int_{\R^d} \bigg(\frac{\rb^\eta}{\rb}\bigg) \big(\Delta \rb + \text{div}(\rb \nabla \rb)\big) \de x.
    \end{align*}
    Integration by parts and rearranging the terms yields 
    \begin{align*}
    &\frac{\de }{\de t}  \mathcal{H}_1(\rb^\eta|\rb) = - \frac{1}{2}\int_{\R^d} \Big(\nabla \rb^\eta \nabla \log\bigg(\frac{\rb^\eta}{\rb}\bigg) - \nabla \bigg(\frac{\rb^\eta}{\rb}\bigg) \nabla \rb \Big)\de x\\
    &+ \frac{1}{2} \int_{\R^d} (\rb^\eta \nabla V^\eta \ast \rb^\eta  - \rb^\eta \nabla \rb) \nabla \log\bigg(\frac{\rb^\eta}{\rb}\bigg) \de x \\
    &= - \frac{1}{2}\int_{\R^d}\rb^\eta \bigg| \nabla \log\bigg(\frac{\rb^\eta}{\rb}\bigg)\bigg|^2 \de x + \frac{1}{2} \int_{\R^d} (\rb^\eta \nabla V^\eta \ast \rb^\eta  - \rb^\eta \nabla \rb) \nabla \log\bigg(\frac{\rb^\eta}{\rb}\bigg) \de x,
    \end{align*}
    where we used that $\rb^\eta \nabla \log(\rb^\eta/\rb) = \rb\nabla (\rb^\eta/\rb)$ in the second line as well as \begin{align*}
    \rb^\eta |\nabla \log(\rb^\eta /\rb)|^2 &= \nabla \log(\rb^\eta /\rb)\bigg(\frac{\rb \nabla \rb^\eta - \rb^\eta \nabla \rb }{\rb}\bigg)= \nabla \rb^\eta \cdot \nabla \log(\rb^\eta /\rb) -  \nabla \bigg(\frac{\rb^\eta}{\rb} \bigg)\cdot \nabla \rb
    \end{align*}
    in the last line. Hence, by Young's inequality and using that $\rho^\eta$ is non-negative we get
    \begin{align*}
    	&\frac{\de }{\de t}  \mathcal{H}_1(\rb^\eta|\rb)=- \frac12\int_{\R^d} \rb^\eta \Big|\nabla\log\Big(\frac{\rb^\eta}{\rho}\Big)\Big|^2\de x +\frac12\int_{\R^d}\rb^\eta\nabla\log\bigg(\frac{\rb^\eta}{\rho}\bigg)(V^\eta*\nabla\rb^\eta-\nabla\rb)\de x
    	\\
    	\leq & -\frac14 \int_{\R^d} \rb^\eta\Big|\nabla\log \Big(\frac{\rb^\eta}{\rho}\Big)\Big|^2 \de x+C\int_{\R^d}\rb^\eta|V^\eta*\nabla\rb^\eta-\nabla\rb^\eta+\nabla\rb^\eta-\nabla\rb|^2 \de x.
    \end{align*}
    Integration over $(0,t)$ for $t>0$ implies that
    \begin{align} \label{H_rb_eta_rb_first}
    	&	\mathcal{H}_1(\rb^\eta|\rb)(t)+\frac14 \int_{0}^{t}\int_{\R^d} \rb^\eta(s)\Big|\nabla\log\Big(\frac{\rb^\eta}{\rb}(s)\Big)\Big|^2 \de x \de s\nonumber\\
    	\leq & C(T)\Big(\int_{0}^t \int_{\R^d} \rb^\eta(s)|V^\eta \ast \nabla \rb^\eta(s) - \nabla \rb^\eta(s) |^2\de x \de s 	\nonumber\\
    	&\hspace{4cm}
    	+ \|\rb^\eta\|_{L^\infty(0,T;L^\infty(\R^d))}\|\nabla\rb^\eta-\nabla\rb\|^2_{L^2(0,T;L^2(\R^d))}\Big).
    \end{align}
    The first integral on the right-hand side can be estimated by Taylor's expansion in the following way:
    \begin{align*}
     &\int_{\R^d} \rb^\eta(s)|V^\eta \ast \nabla \rb^\eta(s) - \nabla \rb^\eta(s) |^2\de x \\
     &=\int_{\R^{d}} \rb^\eta(s,x)\Big|\int_{\R^{d}} V^\eta(x-y)\big(\nabla \rb^\eta(s,y)-\nabla\rb^\eta(s,x)\big) \de y\Big|^2 \de x\\
     &\leq C\int_{\R^d} \rb^\eta(s,x)\Big|\int_{\R^{d}} V^\eta(x-y)\big(D^2\rb^\eta(s,x)(x-y) + R_2(s,y)(x-y)^2\big) \de y\Big|^2 \de x,
    \end{align*}
    where $|R_2(s,y)|\leq C\Vert D^3 \rb^\eta\Vert_{L^\infty(0,T; L^\infty(\R^d))}$ denotes the remainder of Taylor's approximation of order 2 and we use the notation $D^2 f(x) = \sum_{|\alpha|=2}^{}D^\alpha f(x)$.
    Due to $$\int_{\R^{d}} |y|^2 V^\eta(y) \de y \leq C \eta^2$$ as well as the symmetry of $V$, which implies that the function $y \mapsto yV^\eta(y)$ is odd, we get
    \begin{align*}
    &\int_{\R^d} \rb^\eta(s)|V^\eta \ast \nabla \rb^\eta(s) - \nabla \rb^\eta(s) |^2\de x \\&\leq C\Vert D^3 \rb^\eta\Vert_{L^\infty(\R^d)}\int_{\R^{d}} \rb^\eta(s)\Big|\int_{\R^{d}} V^\eta(x-y)|x-y|^2\de y\Big|^2 \de x \\
    &\leq C\Vert D^3 \rb^\eta\Vert_{L^\infty(0,T; L^\infty(\R^d))} \eta^4.
    \end{align*}
    Hence, by \eqref{H_rb_eta_rb_first} and using that $H^s(\R^d)\hookrightarrow L^\infty(\R^d)$, the uniform estimates of $\rb^\eta$ in \eqref{error_rb_rb_eta} yield
    \begin{align}
    	&	\mathcal{H}_1(\rb^\eta|\rb)(t)+\frac14 \int_{0}^{t}\int_{\R^d} \rb^\eta(s)\Big|\nabla\log\Big(\frac{\rb^\eta}{\rb}(s)\Big)\Big|^2 \de x \de s\leq  C(T)\eta^4,
    \end{align}
    where $C>0$ depends on $\|\rb^\eta\|_{L^\infty(0,T;W^{3,\infty}(\R^d))}$ which is uniformly bounded independent of $\eta>0$. This finishes the proof of \eqref{eq:reEn}. \\
    
    Inequality \eqref{eq:error_approx} follows directly as a consequence from the definition of $\rho_{N,L}$, Lemma \ref{lemmassymp} and \eqref{error_rb_rb_eta}:
    \begin{align*} 
    &\|\rb^\eta-\rho_{N,L}\|_{L^\infty(0,T;H^s(\R^d))}+\|\rb^\eta-\rho_{N,L}\|_{L^2(0,T;H^{s+1}(\R^d))} \\&\leq C\big(1+\sum_{l=1}^L\Vert \rho_l \Vert_{L^\infty(0,T;H^s(\R^d))} + \Vert \rho_l \Vert_{L^2(0,T;H^{s+1}(\R^d))} \big)\eta^2 \leq C(T) \eta^2.
    \end{align*}
    This finishes the proof of Lemma \ref{lem.E}.
    \end{proof}
    
    \section{Proof of the main result} \label{sec:main_proof} 
     In this section, we present the relative entropy estimate between $\rho_N^\eta$ and the product measure $(\rb^\eta)^{\otimes N}$, where $\rb^\eta$ solves \eqref{eq:3}. This is the crucial step for the proof of the propagation of chaos result in Theorem \ref{mainthm}, which will be given at the end of this section. Obviously, the chaotic law $\rb_N^\eta: = (\rb^\eta)^{\otimes N}$
    satisfies 
    \begin{equation}\label{eq:li2}
        \partial_t \rb_N^\eta = \frac12 \sum_{i=1}^N \Delta_{x_i} \rb_N^\eta+ \frac12 \sum_{i=1}^N \text{div}_{x_i} \left(\rb_N^\eta  \nabla V^\eta*\rb^\eta (x_i)\right) ,\quad \rb_N^\eta \big|_{t=0} = (\rho_0)^{\otimes N}.
    \end{equation}
    
    The estimate for the difference $\rho_N^\eta-\rb_N^\eta$  in $L^1$-norm is deduced by using the relative entropy estimate and the Csisz\'ar-Kullback-Pinsker inequality. First, let us recall that the empirical measure for particle system \eqref{eq:1} is given by $\mu_N(t):= N^{-1}\sum_{i=1}^N \delta_{X_{i,t}^{N, \eta}(t)}.$

    The key difficulty and main technical step in the mean-field limit discussion in \cite{bresch2019mean,jabin2018quantitative, jabinWang2016} is a so-called \textit{large deviation type estimate}, which implies an estimate of the form
    \begin{align}\label{main_proof_H_ineq}
    \mathcal{H}(\rho_N^\eta|\rb_N^\eta) (t) \leq C(T)\Big( \mathcal{H}(\rho_N^\eta|\rb_N^\eta) (0) + \frac 1 {N^{\theta}}\Big)
    \end{align}
    for some $\theta>0$.
    
    In both cases (relative entropy or modulated free energy) the well-known Csisz\'ar-Kullback-Pinsker inequality then leads to strong $L^1$ convergence.
    
    In this article, in order to get the $L^\infty(0,T;L^1(\R^d))$ convergence for $\rho_N^\eta - \rb^{\otimes N}$, we only need to work with estimates for the relative entropy, where a large deviation estimate is not needed. Instead the following smoothed $L^2(\R^d)$ estimate obtained in \cite{oeschlager1987fluctuation} plays the key role in order to derive \eqref{main_proof_H_ineq} directly.
    \begin{lemma}[Mollified $L^2$ estimate Theorem 1 and Corollary 2 in \cite{oeschlager1987fluctuation}]
    	\label{lemL2}   Let the assumptions in Theorem \ref{mainthm} hold, then there exist constants $C>0$ and $0<\varepsilon < \frac{1}{2} - \beta \frac{d+2}{d}$ independent of $N$ 
    	such that 
    	\begin{align}\label{eq:K2}
    		\mathbb{E}&\Big[\sup_{t\in [0,T]}\big\|W^\eta* (\mu_N-\rho_{N,L})\big\|^2_{L^2(\R^d)}\Big] + \int_0^T \mathbb{E}\Big[ \big\|\nabla W^\eta* (\mu_N-\rho_{N,L})\big\|^2_{L^2(\R^d)}\Big] \de t\nonumber\\
    		&+ \int_{0}^T \mathbb{E}\Big[\langle \mu_N, |\nabla V^\eta \ast (\mu_N - \rho_{N,L})|^2 \rangle \Big]\de t \leq  \frac{C}{N^{\frac{1}{2}+\varepsilon}},
    	\end{align}
    	where $\rho_{N,L}$ denotes the asymptotic approximation of $\rb$ in Lemma \ref{lemmassymp}.
    \end{lemma} 
    \begin{remark}[Modulated free energy]
    In the previous works in \cite{jabin2018quantitative} for $W^{-1,\infty}$ kernels and \cite{bresch2019mean} for logarithmic singular kernels, quantitative mean-field estimates are derived either for the total relative entropy $\mathcal{H}$ or the so-called \textit{modulated free energy $\mathcal{E}$} defined as $$\E(\rho_N^\eta | \rb_N^\eta):= \mathcal{H}(\rho_N^\eta|\rb_N^\eta) + \mathcal{K}(\rho_N^\eta|\rb_N^\eta),$$ where -- using the notation $\nu_{N,y}(\cdot) = N^{-1}\sum_{i=1}^N \delta_{y_i}(\cdot)$ for $y=(y_1,\ldots, y_N) \in \R^{dN}$ -- 
    \begin{align}\label{eq:rE}
    	&\mathcal{H}(\rho_N^\eta|\rb_N^\eta) (t)= \frac{1}{N}\int_{\mathbb{R}^{dN}} \rho_N^\eta\log \frac{\rho_N^\eta}{\rb_N^\eta} \de x_1\ldots \de x_N\\
    	&\mathcal{K}(\rho_N^\eta|\rb_N^\eta)(t) := \frac12\mathbb{E} \Big[ \int_{\R^{2d}} V^\eta(x-y)\de (\mu_N-\rb^\eta)(x) \de (\mu_N-\rb^\eta)(y)\Big] \nonumber \\
    	&= \frac12\int_{\R^{dN}} \rho^\eta_N(z_1, \ldots, z_N) \int_{\R^{2d}} V^\eta(x-y)\de (\nu_{N,z}-\rb^\eta)(x) \de (\nu_{N,z}-\rb^\eta)(y) \de z_1, \ldots, \de z_N.
    \end{align} 
    $\mathcal{K}$ corresponds to the so-called \textit{modulated energy} used in \cite{serfaty2020mean} and \cite{duerinckx16} for mean-field convergence in case of Coulomb or Riesz interactions and $\mathcal{H}$ corresponds to the classical (total) relative entropy. At this point we want to remark that since the articles \cite{jabin2018quantitative,bresch2019mean} are concerned with weak interaction, those estimates would correspond to a mean-field convergence of particle system \eqref{eq:1} towards the intermediate PDE \eqref{eq:3} for fixed $\eta>0$.
    Lemma \ref{lemL2} directly gives an estimate of the modulated energy in the following sense
    \begin{align}\label{eq:K1}
    	\mathcal{K}(\rho_N^\eta|\rho_{N,L}^{\otimes N}) = \frac{1}{2} \mathbb{E}\Big[\big\|W^\eta* (\mu_N-\rho_{N,L})\big\|^2_{L^2(\R^d)}\Big] =o\Big(N^{-\frac{1}{2}}\Big).
    \end{align}
    At this point, the convolution structure of the moderate interacting kernel plays an important role. The rate $o(N^{-\frac{1}{2}})$ should be heuristically optimal since it matches the scaling of central limit theorem. 
    \end{remark}
    \begin{remark}
    It is interesting to notice that convergence in $L^2(\R^d)$-norm for the \\ smoothed empirical measure, i.e. \eqref{eq:K2}, on the one hand implies a central limit theorem as shown in \cite{oeschlager1987fluctuation} and on the other hand in this article it can be used in order to show convergence in relative entropy (without explicitly proving large deviation estimates). This underlines the strength of the technical estimate in \cite{oeschlager1987fluctuation}.
    \end{remark}
    
    Next, we derive the relative entropy estimate by direct application of Lemma \ref{lemL2}. This, combined with the convergence result of \eqref{eq:3} towards \eqref{eq:2} as $\eta\to 0$, implies immediately Theorem \ref{mainthm}. 
    
    \begin{lemma}\label{lm.N}
    Let the assumptions of Theorem \ref{mainthm} hold. Then, there exist constants $C(T)>0$ and $0 < \eps < 1/2-\beta \frac{d+2}{d}$ independent of $N$ such that
    	\begin{align}
    		\sup_{0\leq t\leq T}\mathcal{H}(\rho_N^\eta|\rb_N^\eta)(t) \le \frac{C(T)}{N^{\min\{\frac12+\varepsilon,\frac{4\beta}{d}\}}},
    	\end{align}
    	where $\rho_N^\eta$ denotes the joint law of $(X_{i,t}^{N, \eta})_{i=1}^N$ and $\rb_N^\eta = (\rb^\eta)^{\otimes N}$ denotes the joint law of $(\bar{X}_{i,t}^{N, \eta})_{i=1}^N$. 
    \end{lemma}
    \begin{proof}
    	The evolution of the relative energy is given in the following by a direct computation (similar to the proof of Lemma \ref{lem.E}):
    	Using equations \eqref{eq:li1} and \eqref{eq:li2} for $\rho_N^\eta$ and $\rb_N^\eta$ respectively leads to
    	\begin{align*}
    	&\frac{\de}{\de t} \mathcal{H}(\rho_N^\eta | \rb_N^\eta) = \frac{1}{N}\int_{\R^{dN}} \partial_t \rho_N^\eta \log \bigg(\frac{\rho_N^\eta}{\rb_N^\eta} \bigg) \de x_1 \ldots \de x_N \\
    	&+ \frac{1}{N} \int_{\R^{dN}} \bigg(\partial_t \rho_N^\eta - \frac{\rho_N^\eta \partial_t \rb_N^\eta}{\rb_N^\eta}  \bigg)\de x_1 \ldots \de x_N\\
    	&= \frac{1}{2N} \int_{\R^{dN}} \Bigg[\sum_{i=1}^N \Delta_{x_i} \rho_N^\eta + \sum_{i=1}^N \diver_{x_i}\Big( \rho_N^\eta \frac{1}{N} \sum_{j=1}^N \nabla V^\eta(x_i - x_j)\Big)\Bigg]\log \bigg(\frac{\rho_N^\eta}{\rb_N^\eta} \bigg) \de x_1 \ldots \de x_N\\
    	&- \frac{1}{2N}\int_{\R^{dN}} \frac{\rho_N^\eta}{\rb_N^\eta}\Bigg[ \sum_{i=1}^N \Delta_{x_i} \rb_N^\eta + \sum_{i=1}^N\diver_{x_i}\Big(\rb_N^\eta \nabla V^\eta \ast \rb^\eta \Big)\Bigg] \de x_1 \ldots \de x_N
    	\end{align*}
    	Since
    	\begin{align*}
    	\rho^\eta_N |\nabla_{x_i} \log(\rho_N^\eta /\rb^\eta_N)|^2 &= \nabla_{x_i} \rho^\eta_N \cdot \nabla_{x_i} \log(\rho^\eta_N /\rb^\eta_N) -  \nabla_{x_i} \bigg(\frac{\rho^\eta_N}{\rb_N^\eta} \bigg)\cdot \nabla_{x_i} \rb_N^\eta,
    	\end{align*}
    	 we can reformulate the evolution of the relative energy in the following way
    	 \begin{align*}
    	 &\frac{\de }{\de t}\mathcal{H}(\rho_N^\eta | \rb_N^\eta) = -\frac{1}{2N} \int_{\R^{dN}}\sum_{i=1}^N \left|\nabla_{x_i} \log \bigg(\frac{\rho^\eta_N}{\rb^\eta_N}\bigg)\right|^2 \rho^\eta_N \de x_1\ldots \de x_N \nonumber\\
    	 &- \frac{1}{2N} \sum_{i=1}^N\int_{\R^{dN}}\nabla_{x_i} \log \bigg(\frac{\rho^\eta_N}{\rb^\eta_N}\bigg)\cdot \\&\hspace{2cm}\cdot\bigg(\frac{1}{N}\sum_{j=1}^N\nabla V^\eta(x_i-x_j) -\int_{\R^{d}} \nabla V^\eta(x_i-y)\rb^\eta(y) \de y\bigg) ~\rho^\eta_N \de x_1\ldots \de x_N\nonumber.
    	 \end{align*}
    	 Young's inequality and the fact that $\rho_N^\eta$ is non-negative leads to
    	\begin{align}\label{ineq:part_H_N}
    		&\frac{\de }{\de t} \mathcal{H}(\rho^\eta_N|\rb^\eta_N) \leq -\frac{1}{4N} \int_{\R^{dN}}\sum_{i=1}^N \left|\nabla_{x_i} \log \frac{\rho^\eta_N}{\rb^\eta_N}\right|^2 \rho^\eta_N \de x_1\ldots \de x_N \nonumber\\
    				&+ \frac{1}{4N} \sum_{i=1}^N\int_{\R^{dN}}\bigg|\frac{1}{N}\sum_{j=1}^N\nabla V^\eta(x_i-x_j) -\int_{\R^{d}} \nabla V^\eta(x_i-y)\rb^\eta(y) \de y\bigg|^2 ~\rho^\eta_N \de x_1\ldots \de x_N.
    	\end{align}
    In order the get the desired estimate for the relative entropy, we need an estimate for the second term on the right-hand side. A straightforward computation gives
     \begin{align*}
     &\frac{1}{4N} \sum_{i=1}^N\int_{\R^{dN}}\bigg|\frac{1}{N}\sum_{j=1}^N\nabla V^\eta(x_i-x_j) -\int_{\R^{d}} \nabla V^\eta(x_i-y)\rb^\eta(y) \de y\bigg|^2 ~\rho^\eta_N \de x_1\ldots \de x_N\\
     &= \frac{1}{4}\mathbb{E}\Big(\Big\langle \mu_N, \big|\nabla V^\eta \ast (\mu_N - \rb^\eta)\big|^2 \Big\rangle\Big),
     \end{align*}
     where $\langle \cdot, \cdot \rangle $ is understood in the weak sense.
    	Therefore, by integration over $(0,t)$ of \eqref{ineq:part_H_N} we obtain the following inequality for the relative entropy
    \begin{align}\label{eq:Ht}
    	&\mathcal{H}(\rho^\eta_N|\rb^\eta_N)(t)+\frac{1}{4N} \int^t_0\int_{\R^{dN}}\sum_{i=1}^N \left|\nabla_{x_i} \log \frac{\rho^\eta_N}{\rb^\eta_N}\right|^2 \rho^\eta_N \de x_1\ldots \de x_N \de s\\
    	\leq &\mathcal{H}((\rho_0)^{\otimes N}|(\rho_0)^{\otimes N})+ \frac{1}{4}\int_0^t\mathbb{E}\Big(\Big\langle \mu_N, \big|\nabla V^\eta \ast (\mu_N - \rb^\eta)\big|^2 \Big\rangle\Big) \de s\nonumber\\
    		 \leq &\frac{1}{2}\int_{0}^t\mathbb{E}\Big(\Big\langle \mu_N, \big|\nabla V^\eta \ast (\mu_N - \rho_{N,L})\big|^2 \Big\rangle\Big) \de s+ \frac{1}{2}\int_{0}^t\mathbb{E}\Big(\Big\langle \mu_N, \big|\nabla V^\eta \ast (\rho_{N,L} - \rb^\eta)\big|^2 \Big\rangle\Big) \de s\nonumber \\
    		  \leq & \frac{C}{N^{\frac12+\eps}}+ C\eta^4\leq \frac{C}{N^{\min\{\frac12+\varepsilon,\frac{4\beta}{d}\}}}\nonumber,
    \end{align}
    where we used Lemma \ref{lemL2}, the fact that
    \begin{align*}
    &\int_{0}^t\mathbb{E}\Big(\Big\langle \mu_N, \big|\nabla V^\eta \ast (\rho_{N,L} - \rb^\eta)\big|^2 \Big\rangle\Big) \de s \\
    &\leq \int_{0}^t \Vert \nabla V^\eta \ast (\rho_{N,L} - \rb^\eta ) \Vert_{L^\infty(\R^d)}^2 \de s\leq \Vert V^\eta \Vert_{L^1}^2\int_{0}^t \Vert \nabla (\rho_{N,L} - \rb^\eta )\Vert_{L^\infty(\R^d)}^2 \de s \\ &\leq C\int_{0}^t\Vert \rho_{N,L} - \rb^\eta  \Vert_{H^{s+1}(\R^d)}^2  \de s\leq C\eta^4
    \end{align*}
    for $s>d/2$, as well as the $L^2(0,T;H^{s+1}(\R^d))$ error estimates in Lemma \ref{lem.E} in the last step.
    \end{proof}
    
    Finally, we can prove our main theorem:
    \begin{proof}[Proof of Theorem \ref{mainthm}]
    	Using the super-additivity of the relative entropy as well as the fact that $\rb_N^\eta$ is a product measure on $\R^{dN}$, Lemma \ref{lem:superadd} shows that for all $ 1 \leq k \leq N$
    	$$
    	\frac{1}{k}\mathcal{H}_k(\rho_{N,k}^\eta|(\rb^\eta)^{\otimes k}) = \frac{1}{k}\int_{\R^{dk}} \rho^\eta_{N,k} \log \frac{\rho_{N,k}^\eta}{(\rb^\eta)^{\otimes k}} \de x_1 \ldots \de x_k\le 2\mathcal{H}(\rho^\eta_N|\rb_N^\eta),
    	$$ 
    	where we recall that $\rho_{N,k}^\eta$ denotes the $k$-th marginal of $\rho_{N}^\eta$.
    	Consequently, by the Csisz\'ar-Kullback-Pinsker inequality (Lemma \ref{lem:ckp}), one obtains from Lemma \ref{lm.N} immediately the following estimate for the difference of the $k$-th marginals $\rho_{N,k}^\eta-(\rho^\eta)^{\otimes k}$ 
    	\begin{align*}
    		\|\rho_{N,k}^\eta-(\rb^\eta)^{\otimes k}\|_{L^1(\R^{dk})}^2 \le 2 \mathcal{H}_k(\rho_{N,k}^\eta|(\rb^\eta)^{\otimes k}) \le 4k\mathcal{H}(\rho^\eta_N|\rb_N^\eta)\le \frac{Ck}{N^{\min\{\frac12+\varepsilon,\frac{4\beta}{d}\}}}.
    	\end{align*}
    Thus the results in Theorem \ref{mainthm} follow directly by taking into account the estimate in \eqref{eq:reEn} and the Csisz\'ar-Kullback-Pinsker inequality, i.e.
    	\begin{align}\label{eq:convbar}
    		\|\rb^\eta - \rb\|^2_{L^1(\mathbb{R}^d)} \le 2 \mathcal{H}_1(\rb^\eta|\rb)\le C \eta^4\leq \frac{C}{N^{\min\{\frac12+\varepsilon,\frac{4\beta}{d}\}}},
    	\end{align}
    	and for general $k\geq 1$
    		\begin{align}\label{eq:convbar_k}
    			\|(\rb^\eta)^{\otimes k} - \rb^{\otimes k }\|^2_{L^1(\mathbb{R}^{dk})} \le 2 \mathcal{H}_k((\rb^\eta)^{\otimes k}|\rb^{\otimes k})\leq 2 k \mathcal{H}_1(\rb^\eta|\rb) \leq\frac{C k}{N^{\min\{\frac12+\varepsilon,\frac{4\beta}{d}\}}},
    		\end{align}
    for $0<\eps < 1/2- \beta (d+2)/d$. Therefore we obtain that
    \begin{equation*}
    \|\rho^\eta_{N,k} - \rb^{\otimes k}\|_{L^1(\mathbb{R}^{dk})}\leq \|\rho^\eta_{N,k} - (\rb^\eta)^{\otimes k}\|_{L^1(\mathbb{R}^{dk})}+\|(\rb^\eta)^{\otimes k} - \rb^{\otimes k}\|_{L^1(\mathbb{R}^{dk})}\leq \frac{C k^{1/2}}{N^{\min\{\frac14+\varepsilon,\frac{2\beta}{d}\}}},
    \end{equation*}
    for $0<\eps < 1/4- \beta (d+2)/2d$.
    \end{proof}

    \section{Application to Coulomb interaction}\label{sec:Coulomb}
    In this section, we extend the method of closing the relative entropy estimate on $\R^d$ in \eqref{ineq:part_H_N} to moderately interacting systems with (both attractive and repulsive) Coulomb potentials in dimension $d\geq 2$. In the discussion in Section \ref{sec:main_proof}, the result of $L^2$ estimate given by \cite{oeschlager1987fluctuation} plays the important role in closing the relative entropy estimate. Actually, the convergence of probability in the sense given by \cite{LP} for deriving Vlasov-Poisson equation or in \cite{LCZ} for deriving diffusion system will also help in closing the relative entropy estimate. This will be explained in detail in the following. 
    
    More precisely, we consider the case of particle system \eqref{eq:1} with $V^\eta=\kappa\Phi^\eta$ (aggregation ($\kappa=1$) or repulsion ($\kappa =-1$)), where $\nabla\Phi = \frac{x}{|x|^{d}}$ and $\Phi^\eta$ is a regularization of $\Phi$ which satisfies
    \begin{align}\label{est:Phieta}
    \|D^k\Phi^\eta\|_{L^\infty(\R^d)}\leq C \eta^{-(d-2+k)} ~~~\forall k \in \N.
    \end{align}
    Now in contrast to the sections before, $V^\eta$ does not approximate a Dirac kernel but a kernel $\Phi$ of Coulomb type. An approximation fulfilling \eqref{est:Phieta} can be found for instance in \cite[Lemma 4.14]{Hol23} with use of a cut-off in \cite[Lemma 6.1]{LP}.
    For $\eta=N^{-\beta/d}$ with $\beta\in (0,1/4)$, we will prove the propagation of chaos result from \eqref{eq:1} (with $V^\eta=\kappa\Phi^\eta$) to the following diffusion equation with Coulomb type non-local drift
    \begin{align}\label{1.ks}
    \partial_t \rb = \sigma \Delta  \rb -\kappa \diver\big( \rb \nabla \Phi \ast \rb\big),
    \end{align}
    
    We still use notation $\rb^\eta$ and $\rb^\eta_N$ to be the solutions of the intermediate problems \eqref{eq:3} and \eqref{eq:li2} respectively with $V^\eta=\kappa\Phi^\eta$. 
    We give the following assumptions for $\rb$ and $\rb_\eta$ which can be obtained by certain assumption on the initial data $\rho_0$. Since it is not the main purpose of this article, we refer the readers to the well-posedness of diffusion-aggregation(repulsion) equations. For example, the result for two dimensional Keller-Segel equations in \cite{BDP}. 
    \begin{assumption} \label{Assump.coulomb}The solutions of \eqref{1.ks} and its intermediate version satisfies:
    	\begin{align}\label{est:solution}
    	&\|\rb\|_{L^\infty(0,T;L^1(\R^d)\cap L^\infty(\R^d))}\leq C, \quad\|\rb^\eta\|_{L^\infty(0,T;L^1(\R^d)\cap L^\infty(\R^d))}\leq C, \nonumber\\ &\|\rb-\rb^\eta\|_{L^\infty(0,T;L^1(\R^d))}\leq C \eta^2, \quad\||D^2 \Phi^\eta| \ast \bar{\rho}^\eta\|_{L^\infty(0,T;L^1(\R^d)\cap L^\infty(\R^d))}\leq C
    	\end{align}
    \end{assumption}
    The following lemma is a law of large numbers estimate for the i.i.d. random variables $(\bar X_{i,t}^{\eta})_{i=1}^N$, which will be used later for the relative entropy estimate:
    \begin{lemma}[Law of Large numbers, \cite{Hol23}]\label{lem.law}
    Let $(\bar X_{i,t}^{\eta})_{i=1}^N$ be the solution to system \eqref{eq:interPar} with $V^\eta = \kappa \Phi^\eta$ and let
    $\bar \rho^\eta$ be the density function associated to $\bar X_{i,t}^{\eta}$. 
    Given
    $\psi_\eta\in L^\infty(\R^d;\R^n)$ with
    $n\in\{1,d,d\times d\}$, we define the random variables
    \begin{align}\label{def.H}
    	h_i(t):=
    	\bigg|\frac{1}{N}\sum_{j=1}^N\psi_\eta\big(\bar X_{i,t}^{\eta}-\bar X_{j,t}^{\eta}\big)
    	- (\psi_\eta*\bar \rho^\eta)(\bar X_{i,t}^{\eta})\bigg|~~~i=1,\ldots,N. 
    \end{align}
    Then, for every $m\in\N$ and $T>0$, there exists $C(m)>0$ such that for all
    $0<t<T$,
    \begin{align*}
    	\mathbb{E}(h_i(t)^{2m}) 
    	&\le C(m)\|\psi_\eta\|_{L^\infty}^{2m}N^{-m}.
    \end{align*}
    \end{lemma} 
    A proof is given in \cite[Lemma 4.2]{Hol23}.
    
    The second ingredient of the proof of the main theorem of this section (Theorem \ref{thm.coulomb}) is a convergence result in probability towards the intermediate particle system \eqref{eq:interPar}. A similar convergence in probability result has been proved by Lazarovici and Pickl for second order system in deriving Vlasov-Poisson system in \cite{LP}. It was extended to first order system with Coulomb interaction in the following lemma, the proof can be found in \cite[Theorem 4.12]{Hol23}. Note that in \cite{Hol23} the notation $\eta= N^{-\beta}$ was used whereas in the present paper we follow \cite{oeschlager1987fluctuation} and use $\eta=N^{-\beta/d}$.
    \begin{lemma}[Convergence in probability to the intermediate system \ref{eq:interPar}]\label{lem.prob}
    Let Assumption \ref{Assump.coulomb} holds. Let $V^\eta=\kappa\Phi^\eta$ with $\eta=N^{-\beta/d}$, then for $0<\beta<1/4$ and $\beta\frac{d+1}{d}<\alpha<1/2-\beta\frac{d-1}{d}$, we have that for any $\gamma>0$ and $T>0$, there exists $C(\gamma,T)>0$ such that for all $0<t<T$,
    	\begin{equation}\label{1.prop}
    	\mathbb{P}\bigg(\mathcal{A}_\alpha=\Big\{\omega\in\Omega:\max_{i=1,\ldots,N}|X_{i,t}^{N,\eta}-\bar X_{i,t}^{\eta}|>N^{-\alpha}\Big\}\bigg)
    	\le C(\gamma,T)N^{-\gamma}.
    	\end{equation}
    \end{lemma}
    The main part of this section is the following theorem, which shows how to combine Lemma \ref{lem.law} and Lemma \ref{lem.prob} to close the relative entropy estimates for Coulomb kernels in the moderate regime:
    
    \begin{theorem}\label{thm.coulomb}
    	Let assumption \eqref{est:solution} holds, then for $0<\beta<1/4$, there exists $C>0$ which does not depend on $N$ such that it holds that
    	\begin{align}
    	\mathcal{H}(\rho^\eta_N|\rb^\eta_N) \leq C(T)N^{-2\beta/d}
    	\end{align}
    	which implies
    	\begin{align}
    	\|\rho^\eta_{N,1}-\rb\|_{L^\infty(0,T;L^1(\R^d))}\leq N^{-2\beta/d},
    	\end{align}
    	where $\rb$ solves \eqref{1.ks}.
    \end{theorem}
    
    \begin{proof}
    Following the same strategy as in the proof of the main theorem (Theorem \ref{mainthm}), in order to close the relative entropy estimate in \eqref{ineq:part_H_N}, we only need to give an estimate for the following term
    \begin{align} \label{estM}
    &\mathbb{E}\Big(\Big\langle \mu_N, \big|\nabla \Phi^\eta \ast (\mu_N - \rb^\eta)\big|^2 \Big\rangle\Big)\nonumber\\ &=\mathbb{E}\Big(\frac{1}{N}\sum_{i=1}^N\Big|\frac{1}{N}\sum_{j=1}^N\nabla\Phi^\eta(X_{i,t}^{N,\eta}-X_{j,t}^{N,\eta})-\nabla\Phi^\eta\ast\rb^\eta(X_{i,t}^{N,\eta})\Big|^2\Big)\nonumber\\
    &\leq \mathbb{E}\Big(\frac{1}{N}\sum_{i=1}^N\Big|\frac{1}{N}\sum_{j=1}^N\nabla\Phi^\eta(X_{i,t}^{N,\eta}-X_{j,t}^{N,\eta})-\frac{1}{N}\sum_{j=1}^N\nabla\Phi^\eta(\bar X_{i,t}^{\eta}-\bar X_{j,t}^{\eta})\Big|^2\Big)\nonumber\\
    &+\mathbb{E}\Big(\frac{1}{N}\sum_{i=1}^N\Big|\frac{1}{N}\sum_{j=1}^N\nabla\Phi^\eta(\bar X_{i,t}^{\eta}-\bar X_{j,t}^{\eta})-\nabla\Phi^\eta\ast\rb^\eta(\bar X_{i,t}^{\eta})\Big|^2\Big) \nonumber \\
      &+\mathbb{E}\Big(\frac{1}{N}\sum_{i=1}^N\Big|\nabla\Phi^\eta\ast\rb^\eta(\bar X_{i,t}^{\eta})-\nabla\Phi^\eta\ast\rb^\eta(X_{i,t}^{N,\eta})\Big|^2\Big) =I_1+I_2+I_3.
    \end{align}
    Since the solutions $\bar X^\eta_{i,t}$ are i.i.d. random variables, by using Lemma \ref{lem.law} (law of large numbers), we can estimate the second term easily:  Indeed, let $h_i(t)$ be defined as in \eqref{def.H} with $\psi_\eta = \nabla \Phi^\eta$, then Lemma \ref{lem.law} yields for $m=1$
    \begin{align*}
    	I_2&= \mathbb{E}\bigg(\frac{1}{N}\sum_{i=1}^N\Big|\frac{1}{N}\sum_{j=1}^N\nabla\Phi^\eta(\bar X_{i,t}^{\eta}-\bar X_{j,t}^{\eta})-\nabla\Phi^\eta\ast\rb^\eta(\bar X_{i,t}^{\eta})\Big|^2\bigg)\\
    	&\leq \dfrac{C\big\|\nabla\Phi^\eta\big\|_{L^\infty(\R^d)}^2}{N}\leq CN^{-1+2\beta(d-1)/d} \leq CN^{-1/2},
    \end{align*}
    where the estimates for regularized $\Phi$ in \eqref{est:Phieta} has been used. The estimates for $I_1$ and $I_3$ will be estimated by using the convergence in probability. In the following we use $I_1$ as the example, the estimate for $I_3$ can be done similarly. Let $\mathcal{A}_\alpha$ be defined as in Lemma \ref{lem.prob}, then it yields
    \begin{align*}
    I_1\leq &\; \mathbb{E}\bigg(\Big(\mathds{1}_{{\Omega \backslash\mathcal{A}_\alpha}}+\mathds{1}_{{\mathcal{A}_\alpha}}\Big)\frac{1}{N}\sum_{i=1}^N\Big|\frac{1}{N}\sum_{j=1}^N\nabla\Phi^\eta(X_{i,t}^{N,\eta}-X_{j,t}^{N,\eta})-\frac{1}{N}\sum_{j=1}^N\nabla\Phi^\eta(\bar X_{i,t}^{\eta}-\bar X_{j,t}^{\eta})\Big|^2\bigg)\\
    \leq & \; \frac{\big\|D^2\Phi^\eta\big\|^2_{L^\infty(\R^d)}}{N^{2\alpha}} +2\|\nabla\Phi^\eta\big\|_{L^\infty(\R^d)}^2\sup_{0<t<T} \mathbb{P}(\mathcal{A}_\alpha)\leq \dfrac{CN^{2\beta}}{N^{2\alpha}} + \dfrac{C(\gamma)N^{2\beta(d-1)}/d}{N^\gamma},
    \end{align*}
    where $\gamma>0$ from Lemma \ref{lem.prob} can be arbitrary large for $\alpha$ and $\beta$ chosen accordingly. Taking into account of the fact that $0<\beta<1/4$ and $\beta(d+1)/d<\alpha<1/2-\beta(d-1)/d$, we know that $2\beta - 2\alpha < 2\beta -2\beta(d+1)/d =-2\beta/d$ and hence $N^{-2\beta/d}$ is the leading order term in the above estimates.
    Therefore, combining this estimate together with \eqref{ineq:part_H_N}, we finish the relative entropy estimate
    \begin{align*}
    	&\frac{\de }{\de t} \mathcal{H}(\rho^\eta_N|\rb^\eta_N) \leq -\frac{1}{4N} \int_{\R^{dN}}\sum_{i=1}^N \left|\nabla_{x_i} \log \frac{\rho^\eta_N}{\rb^\eta_N}\right|^2 \rho^\eta_N \de x_1\ldots \de x_N + \dfrac{C}{N^{2\beta/d}},
    \end{align*}
    which by the Csisz\'ar-Kullback-Pinsker inequality implies that
    \begin{align*}
    \|\rho^\eta_{N,1}-\rb^\eta\|_{L^\infty(0,T;L^1(\R^d))} \leq CN^{-2\beta/d}
    \end{align*}
    Together with assumption \eqref{est:solution}, we get
    \begin{align*}
    \|\rho^\eta_{N,1}-\rb\|_{L^\infty(0,T;L^1(\R^d))} \leq \|\rho^\eta_{N,1}-\rb^\eta\|_{L^\infty(0,T;L^1(\R^d))}  + \|\rb^\eta-\rb\|_{L^\infty(0,T;L^1(\R^d))} \leq C N^{-2\beta/d} ,
    \end{align*}
     the theorem is proved.
    \end{proof}

    \appendix

    \section{Useful inequalities}
    This appendix contains analytical inequalities used in this note. The following two Moser-type inequalities are used for the analysis of the intermediate equation \eqref{eq:3}.
    \begin{lemma}[Moser-type commutator inequality, {\cite[Prop.~2.1(B)]{Maj84}}]
    \label{lem.comm}
    Let $s\in\N$ and $\alpha\in\N_0^n$ with $|\alpha|=s$. Then
    there exists $C>0$ such that for all $f\in H^s(\R^d)\cap W^{1,\infty}(\R^d)$
    and $g\in H^{s-1}(\R^d)\cap L^\infty(\R^d)$,
    $$
      \|D^\alpha(fg)-fD^\alpha(g)\|_{L^2(\R^d)} \le C\big(\|Df\|_{L^\infty(\R^d)}
    	\|D^{s-1}g\|_{L^2(\R^d)} + \|D^s f\|_{L^2(\R^d)}\|g\|_{L^\infty(\R^d)}\big),
    $$
    where $D^s=\sum_{|\alpha|=s}D^\alpha$.
    \end{lemma}
    
    \begin{lemma}[Moser-type estimate I, {\cite[Prop.~2.1(A)]{Maj84}}]\label{lem.moser}
    Let $s\in\N$ and $\alpha\in\N_0^n$ with $|\alpha|=s$. Then there exists
    a constant $C>0$ such that for all $f$, $g\in H^s(\R^d)\cap L^\infty(\R^d)$,
    $$
      \|D^\alpha(fg)\|_{L^2(\R^d)} \le C\big(\|f\|_{L^\infty(\R^d)}\|D^s g\|_{L^2(\R^d)}
    	+ \|D^s f\|_{L^2(\R^d)}\|g\|_{L^\infty(\R^d)}\big).
    $$
    \end{lemma} 
    
    The Csisz\'ar-Kullback-Pinsker inequality is crucial for this article since it shows a connection between convergence in $L^1(\R^d)$-norm and with respect to the relative entropy.
    \begin{lemma}[Csisz\'ar-Kullback-Pinsker inequality, \cite{Villani02}]\label{lem:ckp}
    Let $m \in \N$ be fixed. If $f$ and $g$ are probability density functions on $\R^{dm}$ with respect to the Lebesgue measure, then
    \begin{align*}
    \Vert f - g \Vert_{L^1(\R^{dm})} \leq \sqrt{2 \int_{\R^{dm}} f \log \frac{f}{g} \de x}.
    \end{align*}
    \end{lemma}

    The next lemma gives an important inequality which shows that the scaled relative entropy of any $k$-marginal can be bounded by the total relative entropy. A similar result can be found in \cite[Lemma 3.9]{Miclo01}. We present a short proof for the reader's convenience.
    \begin{lemma}[Super--additivity of relative entropy with respect to product measure] \label{lem:superadd}Let $N \in \N$ be fixed and let $F_N$ be a probability distribution on $\R^{dN}$. By $F_{N,k}$ we denote the $k$-th marginal of $F_N$ for any $k \geq 1$. Furthermore, for any probability density function $\rb^\eta$ on $\R^d$, we define the relative entropy of $F_{N,k}$ w.r.t $(\rb^\eta)^{\otimes k }$ on $\R^{dk}$ as
    \begin{align*}
    \mathcal{H}_k(F_{N,k}|(\rb^\eta)^{\otimes k } ):=\int_{\R^{dk}} F_{N,k} \log \frac{F_{N,k}}{(\rb^\eta)^{\otimes k }} \de x_1 \ldots \de x_k.
    \end{align*}
    Then, it holds for any $k\geq 1$ that
    \begin{align}\label{appendix_superadd}
    \mathcal{H}_k(F_{N,k}|(\rb^\eta)^{\otimes k } ) +  \mathcal{H}_{N-k}(F_{N,{N-k}}|(\rb^\eta)^{\otimes (N-k) } )\leq \mathcal{H}_N(F_N|(\rb^\eta)^{\otimes N}).
    \end{align} 
    Additionally, by denoting with $[r]$ the integer part of any $r\in \R$, this implies
    \begin{align}\label{appendix_N_k_relativ}
    \frac{1}{N} \mathcal{H}_N(F_N|(\rb^\eta)^{\otimes N}) \geq \frac{1}{N}\bigg[\frac{N}{k}\bigg] \mathcal{H}_k(F_{N,k}|(\rb^\eta)^{\otimes k } ) \geq \frac{1}{2k} \mathcal{H}_k(F_{N,k}|(\rb^\eta)^{\otimes k } ).
    \end{align}
    \end{lemma}
    \begin{proof} For simplicity, we drop the argument in the relative entropy and write $\mathcal{H}_\ell:=\mathcal{H}_\ell(F_{N,\ell}|(\rb^\eta)^{\otimes \ell}) $ for any $\ell \leq N$. Similar as the proof in \cite[Lemma 3.3 (iv)]{HaurayMischler14}, we get by the definition of the $k$-th marginal 
    \begin{align*}
    &\mathcal{H}_N - \mathcal{H}_k - \mathcal{H}_{N-k} = \int_{\R^{dN}} F_N(x_1,\ldots, x_N) \log\frac{F_N(x_1,\ldots, x_N)}{(\rb^\eta)^{\otimes N}} \de x_1\ldots \de x_N \\
    &- \int_{\R^{dN}} F_{N}(x_1,\ldots, x_N) \Big[\log\frac{F_{N,k}(x_1,\ldots, x_k)}{(\rb^\eta)^{\otimes k}} + \log \frac{F_{N,{N-k}}(x_{k+1},\ldots, x_{N})}{(\rb^\eta)^{\otimes {N-k}}}\Big]\de x_1\ldots \de x_N \\
    &= \int_{\R^{dN}} F_N(x_1,\ldots, x_N) \log\frac{F_N(x_1,\ldots, x_N)}{F_{N,k}(x_1, \ldots,x_k) \otimes F_{N,N-k}(x_{k+1}, \ldots, x_N)} \de x_1\ldots \de x_N \\
    &\geq 0,
    \end{align*}
    due to the non-negativity of the relative entropy and the fact that $$F_{N,N-k}\otimes F_{N,k},$$ is a density function on $\R^{dN}$.
    
    Assuming that $N=m k$ for $m \in \N$ and iterating this argument for $\mathcal{H}_{N-\ell k}$ leads to 
    \begin{align*}
    \mathcal{H}_N - m\mathcal{H}_k \geq 0,
    \end{align*}
    which is -- since $m=N/k$ -- equivalent to $\frac{1}{N} \mathcal{H}_N \geq \frac{1}{k} \mathcal{H}_k$. 
    If $N=mk + n $ for $n < k $ and $m\in \N$, then we get 
    \begin{align*}
    \mathcal{H}_N - m \mathcal{H}_k - \mathcal{H}_n \geq 0,
    \end{align*}
    and hence with noting that $m=[N/k]$ and the non-negativity of the relative entropy
    \begin{align*}
    \mathcal{H}_N \geq \bigg[\frac{N}{k}\bigg] \mathcal{H}_k + \mathcal{H}_n \geq \bigg[\frac{N}{k}\bigg] \mathcal{H}_k.
    \end{align*}
    Consequently, $\frac{1}{N} \mathcal{H}_N \geq \frac{1}{N}\big[\frac{N}{k}\big] \mathcal{H}_k \geq \frac{1}{2k} \mathcal{H}_k$ since $[r] \geq \frac{1}{2} r$ for any $r \in \R_{\geq 1}$. This finishes the proof.
    \end{proof}

\end{document}